\documentclass[11pt,a4paper]{article}
\usepackage{amsmath}
\usepackage{amsfonts}
\usepackage{amssymb}
\usepackage{centernot}
\usepackage[margin=2.7cm]{geometry}
\usepackage[colorlinks=true,linkcolor=blue,citecolor=blue]{hyperref}
\usepackage{enumitem}
\usepackage{graphicx,fancybox}
\usepackage{tikz}
\usepackage{subfigure}
\usetikzlibrary{positioning}
\usepackage{tkz-berge}
\usepackage{hhline}
\usepackage{multirow}
\usepackage[capitalise]{cleveref}
\usepackage[ruled,vlined]{algorithm2e}
\usepackage{siunitx}
\usepackage{sudoku}
\usepackage{accents}

\usepackage{amsthm}

\newtheorem{proposition}{Proposition}[section]

\newtheorem{fact}{Fact}[section]
\newtheorem{remark}{Remark}[section]
\newtheorem{example}{Example}[section]

\Crefname{formulation}{Formulation}{Formulations}
\Crefname{example}{Example}{Example}
\Crefname{algocf}{Algorithm}{Algorithms}

\newcommand{\diag}{\operatorname{diag}}

\newcommand{\rank}{\operatorname{rank}}

\newcommand{\tr}{\operatorname{tr}}

\newcommand{\Wd}{\operatorname{Wd}}
\newcommand{\E}{\mathbb{E}}
\newcommand{\R}{\mathbb{R}}
\newcommand{\Sy}{\mathcal{S}}
\newcommand{\sdp}{\mathcal{S}_{+}}
\newcommand{\Id}{\operatorname{Id}}

\newcounter{row}
\newcounter{col}

\DeclareFontFamily{U}{mathx}{\hyphenchar\font45}
\DeclareFontShape{U}{mathx}{m}{n}{
      <5> <6> <7> <8> <9> <10>
      <10.95> <12> <14.4> <17.28> <20.74> <24.88>
      mathx10
      }{}
\DeclareSymbolFont{mathx}{U}{mathx}{m}{n}
\DeclareFontSubstitution{U}{mathx}{m}{n}
\DeclareMathAccent{\widecheck}{0}{mathx}{"71}

\makeatletter
\newcommand{\mylabel}[2]{#2\def\@currentlabel{#2}\label{#1}}
\makeatother

\title{An enhanced formulation for solving graph coloring problems with the Douglas--Rachford algorithm}

\author{Francisco J. Arag\'on Artacho\thanks{Department of Mathematics,
University of Alicante, \textsc{Spain}. e-mail:~\url{francisco.aragon@ua.es}}
        \and Rub\'en Campoy\thanks{Department of Mathematics,
University of Alicante, \textsc{Spain}. e-mail:~\url{ruben.campoy@ua.es}}
		\and Veit Elser\thanks{Department of Physics,
		Cornell University, Ithaca, \textsc{NY}. e-mail:~\url{ve10@cornell.edu}}
}

\usepackage{todonotes}

\begin{document}

\maketitle

\begin{abstract}
We study the behavior of the Douglas--Rachford algorithm on the graph vertex-coloring problem. Given a graph and a number of colors, the goal is to find a coloring of the vertices so that all adjacent vertex pairs have different colors. In spite of the combinatorial nature of this problem, the Douglas--Rachford algorithm was recently shown to be a successful heuristic for solving a wide variety of graph coloring instances, when the problem was  cast as a feasibility problem on binary indicator variables. In this work we consider a different formulation, based on semidefinite programming. The much improved performance of the Douglas--Rachford algorithm, with this new approach, is demonstrated through various numerical experiments.

\end{abstract}


\section{Introduction}
Let $G=(V,E)$ be a \emph{graph} with $n$ vertices $V=\{1,\ldots,n\}$ that are connected by edges $E\subset V\times V$. A \emph{proper $m$-coloring} of the graph $G$ is a mapping $c:V\mapsto K:=\{1,\ldots,m\}$, assigning one of the $m$ possible colors to each vertex, such that no two adjacent vertices share the same color, that is,
$$c(i)\neq c(j) \text{ for all } \{i,j\}\in E.$$
The \emph{graph coloring problem} consists in determining whether it is possible to find a proper $m$-coloring of the graph~$G$. Many problems arising from different fields can be formulated as graph coloring problems. Some applications include time--tabling and scheduling~\cite{L79}, computer register allocation~\cite{C04}, radio frequency assignment~\cite{H80}, and printed circuit board testing~\cite{GJS76}. Since the graph coloring problem is proved to be NP-complete~\cite{K72}, most commonly used solvers rely on heuristics. For a basic reference on graph coloring, algorithms and applications, see, e.g.,~\cite{JT95,L16}, or the surveys~\cite{FT12,PMX98}.

The aim of this paper is to study the behavior of the Douglas--Rachford (DR) algorithm when it is applied to the Karger--Motwani--Sudan (KMS) \cite{KMS} semi-definite programming formulation of the graph coloring problem. The KMS formulation is reviewed in \Cref{sec:formulation}.  The DR algorithm belongs to the family of so-called projection algorithms, which arise in convex optimization. Although the convergence of the DR algorithm is only guaranteed in the convex setting, it has been successfully applied to many nonconvex problems, including those of combinatorial type (see, e.g.,~\cite{ABTmatrix,ABTcomb,BKroad,Elser}). However, the theory in nonconvex settings is very limited, with results on global behavior limited to special sets~\cite{ABT16,benoist}. In most applications the constraint sets satisfy some type of regularity property, and local convergence can be proved,  see, e.g.~\cite{BNlocal,HLnonconvex,Plinear}.

This is not the first time that the DR scheme has been used to solve graph coloring problems. It was first employed by Elser et al.~\cite{Elser} for edge-colorings, in particular, colorings that avoid monochromatic triangles. In a more recent paper by Arag\'on and Campoy~\cite{AC18}, the DR algorithm was applied as a heuristic for vertex coloring. In that work, the feasibility problem was expressed in terms of binary indicator variables, the same variables that would be used in an integer programming formulation. This formulation was easily adapted to solve (using DR) variants and generalizations of the graph coloring problem, including list coloring, partial coloring, and finding Hamiltonian cycles.

Our numerical experiments indicate that the KMS formulation appears to be superior to the indicator variable formulation, when using the DR heuristic. While we do not have an interpretation of this result, it is empirically supported on a wide spectrum of problem instances.

The remainder of the paper is organized as follows. In~\Cref{sec:prelim} we recall some preliminary notions and results. The KMS formulation is reviewed in~\Cref{sec:formulation}. In~\Cref{sec:implementation} we give details on the DR implementation. Finally, in~\Cref{sec:numexp} we collect various numerical experiments where we show the good performance of the DR algorithm for the KMS formulation.

\section{Preliminaries}\label{sec:prelim}

Let $\E$ be a finite-dimensional Hilbert space with inner product $\langle\cdot,\cdot\rangle$ and induced norm $\|\cdot\|$. Throughout this paper the space $\E$ will be the Euclidean space $\R^{n\times m}$ of the $n\times m$ real matrices with inner product $\langle A,B\rangle=\tr\left(A^TB\right).$ The induced norm corresponds to the \emph{Frobenius} norm,
$$\|A\|_F=\sqrt{\tr\left(A^TA\right)}=\sqrt{\sum_{i=1}^n\sum_{j=1}^m a_{ij}^2}.$$
We denote by $\mathcal{S}^n$ the set of symmetric matrices in $\R^{n\times n}$, while $\sdp^n$ is the set of positive semidefinite matrices.

The \emph{projector} onto a nonempty closed set $C\subseteq \E$ is the set-valued mapping $P_C:\E\rightrightarrows C$ given by
$$P_C(x):=\left\{p\in C :  \|p-x\|=\inf_{c\in C}\|c-x\|\right\},$$
and the \emph{reflector} is defined as $R_C:=2P_C-\Id$, where $\Id$ denotes the identity operator. Any element $\pi_C(x)\in P_C(x)$ is said to be a \emph{best approximation} or a \emph{projection} of $x$ onto $C$. If $C$ is also convex, then there exists a unique projection. In fact, a closed set in a Hilbert space is convex if and only if its projector is everywhere single-valued (see, e.g.,~\cite[Theorem~3.2]{ABMY14}).

Given $C_1,C_2,\ldots,C_r\subseteq \E$, the \emph{feasibility problem} consists in finding a point belonging to all these sets, that is,
\begin{equation}\label{eq:FeasibilityProblem}
\text{Find } x\in\bigcap_{i=1}^r C_i.
\end{equation}

The DR algorithm is a powerful tool for solving feasibility problems, whenever the individual projectors onto the sets can be easily computed.
The DR scheme is defined for problems involving two sets. Next we recall its main properties in the convex setting.

\begin{fact}\label{fact:convexDR}
Let $A,B\subseteq \E$ be closed and convex sets and let $\lambda\in{]0,2[}$. Consider the DR operator defined by
\begin{equation}\label{eq:DR}
T_{A,B,\lambda}:=\left(1-\frac{\lambda}{2}\right)\Id+\frac{\lambda}{2} R_BR_A.
\end{equation}
Given any $x_0\in \E$,  define $x_{n+1}:=T_{A,B,\lambda}(x_n)$, for every $n\geq0$. Then the following holds:
\begin{itemize}[noitemsep,topsep=5pt]
\item [(i)] If $A\cap B\neq\emptyset$, then $\{x_n\}$ converges to a point $x^\star$ with $P_A(x^\star)\in A\cap B$.
\item [(ii)] If $A\cap B=\emptyset$, then $\|x_n\|\rightarrow +\infty$.
\end{itemize}
\end{fact}
\begin{proof}
(i) See, e.g.,~\cite[Theorem~26.11]{BC17} or~\cite[Corollary~5.2.4]{C12}. (ii) See~\cite[Corollary~2.2]{BBR78}.
\end{proof}

Although there is no guarantee of convergence when DR is applied to general (not necessarily convex) closed sets, it can still be applied as a heuristic. In this framework, observe that the DR operator may be multivalued due to the fact that the projection onto nonconvex sets is not necessarily unique. Therefore, the equality in~\eqref{eq:DR} must be replaced by an inclusion, and the iteration takes the form
\begin{equation}\label{eq:NonconvexDR}
x_{n+1}\in T_{A,B,\lambda}(x_n):=\left\{x_n+\lambda(b_n-a_n)\in \E: a_n\in P_A(x_n), b_n\in P_B(2a_n-x_n)\right\}.
\end{equation}

Finally, we recall that any feasibility problem \eqref{eq:FeasibilityProblem} can be reduced to a two-set problem through Pierra's \emph{product space reformulation}~\cite{Pierra}. This permits us to employ the scheme~\eqref{eq:DR} for solving feasibility problems involving more than two sets. The reformulation relies on the equivalence
\begin{equation*}
	x\in\bigcap_{i=1}^r C_i\subseteq \E \Leftrightarrow (x,x,\ldots,x)\in C\cap D\subseteq \E^r,
\end{equation*}
where the constraint sets $C$ and $D$ are defined as
\begin{equation*}
C:=C_1\times C_2\times\ldots\times C_r\quad\text{and} \quad D:=\left\{(x,x,\ldots,x)\in \E^r:x\in \E\right\}.
\end{equation*}
Moreover, the projectors onto $C$ and $D$ are readily computable as long as the projectors onto $C_1,C_2,\ldots,C_r$ are. Indeed, for any $\mathbf{x}=(x_1,x_2,\ldots,x_r)\in \E^r$, we have
\begin{gather*}
P_C(\mathbf{x})=P_{C_1}(x_1)\times P_{C_2}(x_2)\times\ldots\times P_{C_r}(x_r),\\
P_D(\mathbf{x})=\left( \frac{1}{r}\sum_{i=1}^r x_i, \frac{1}{r}\sum_{i=1}^r x_i,\ldots, \frac{1}{r}\sum_{i=1}^r x_i \right);
\end{gather*}
see~\cite[Lemma~1.1]{Pierra}. For further details see, e.g.,~\cite[Section~3]{ABTcomb}.

\section{Coloring graphs with vertices of the regular simplex}\label{sec:formulation}

Karger, Motwani and Sudan~\cite{KMS} proposed using the geometry of the regular simplex to formulate the graph vertex-coloring problem. This geometrical encoding of the problem, which respects all its symmetries, is well suited to projection based algorithms.

Suppose that we have a proper $m$-coloring of the graph $G=(V,E)$ given by  $c:V\mapsto K$. The $m$-coloring $c$ can be represented by a matrix as follows.
Let $u_1, u_2, \ldots, u_{m}\in\R^{m-1}$ be the vertices of a standard centered regular $(m-1)$-simplex.
Then, $\{u_1, u_2, \ldots, u_{m}\}$ are $m$ unit vectors whose pairwise dot products are equal to $\frac{-1}{m-1}$, since
\begin{equation*}
\langle u_1+\cdots+u_m, u_1+\cdots+u_m\rangle=0.
\end{equation*}
Each of the vertices of the $(m-1)$-simplex shall represent one of the $m$ colors. Hence, the $m$-coloring of the graph $G$ can be recovered from the matrix
\begin{equation}\label{eq:Uc0}
U_c:=\left[u_{c(1)},u_{c(2)},\ldots,u_{c(n)}\right]\in\R^{(m-1)\times n},
\end{equation}
whose rows are vertices of the $(m-1)$-simplex (possibly repeated). Finally, let us construct the \emph{Gram matrix} associated with $W_c$, namely,
\begin{equation}\label{eq:Uc}
W_c:=U_c^TU_c\in\R^{n\times n}.
\end{equation}
One can easily check that $W_c$ satisfies the following properties (see, e.g.~\cite[Theorem~7.2.10]{HJ13}):
\begin{itemize}
\item[\mylabel{(P1)}{(P1)}] $W_c\in\sdp^n$,
\item[\mylabel{(P2)}{(P2)}] $\rank(W_c)\leq m-1$,
\item[\mylabel{(P3)}{(P3)}] $W_c\in\left\{1,\frac{-1}{m-1}\right\}^{n\times n}$ and some of the entries of $W_c=[w_{ij}]$ are determined as follows:
\begin{equation*}
w_{ii}=1,\ \forall i\in V, \quad\text{and}\quad
w_{ij}=\frac{-1}{m-1},\ \forall \{i,j\}\in E.
\end{equation*}
\end{itemize}

Therefore, every valid $m$-coloring of the graph $G$ leads to a matrix having properties~\ref{(P1)}--\ref{(P3)}. In fact, this is an equivalence, as we shall show after the next illustrative example.

\begin{example}\label{example}
Consider a graph $G=(V,E)$ where the set of vertices is $V=\{1,2,3,4,5\}$, and the set of edges is $E=\left\{\{1,2\},\{1,3\},\{2,3\},\{2,4\},\{3,5\}\right\}$. A proper $3$-coloring of $G$ is shown in~\Cref{fig:example_a}. We identify each of the colors with one of the vertices $u_1,u_2,u_3\in\R^2$ of a standard centered regular $2$-simplex (see~\Cref{fig:example_b}), where
$$u_1=\left(1,0\right)^T, \quad u_2=\frac{1}{2}\left(-1,\sqrt{3}\right)^T \quad and \quad u_3=\frac{1}{2}\left(-1,-\sqrt{3}\right)^T.$$
Then the matrix representation of $c$ given in~\eqref{eq:Uc} becomes
{\def\arraystretch{1.5}
\begin{equation*}
W_c=U_c^TU_c=\left(\begin{array}{ccccc} \boxed{1} & \boxed{-0.5} & \boxed{-0.5} & 1 & -0.5\\
\boxed{-0.5} & \boxed{1} & \boxed{-0.5} & \boxed{-0.5} & 1\\
\boxed{-0.5} & \boxed{-0.5} & \boxed{1} & -0.5 & \boxed{-0.5}\\
1 & \boxed{-0.5} & -0.5 & \boxed{1} & -0.5\\
-0.5 & 1 & \boxed{-0.5} & -0.5 & \boxed{1} \end{array}\right),\quad \text{with } U_c=\left[u_1,u_2,u_3,u_1,u_2\right].
\end{equation*}}%
The boxed entries in $W_c$ correspond to those determined by~\ref{(P3)}.

\begin{figure}[ht!]
	\centering
	\subfigure[A $3$-coloring of the graph\label{fig:example_a}]{\hspace{4ex}
	\begin{tikzpicture}[scale=.85,transform shape]%
		\def\sep{1.25}
		\GraphInit[vstyle=Normal]
		\SetVertexNoLabel
		\tikzstyle{EdgeStyle}=[thick]
		\coordinate (A) at (-5,0);
		\coordinate (B) at (5,0);
		\Vertex[x=-\sep, y=-\sep] {c0}
		\Vertex[x=-\sep, y=\sep] {c1}
		\Vertex[x=\sep, y=-\sep] {c2}
		\Vertex[x=\sep, y=\sep] {c3}
		\Vertex[x=2*\sep, y=0] {c4}
		\AddVertexColor{red!50}{c0,c3}
		\AddVertexColor{green!50}{c1,c4}
		\AddVertexColor{blue!40}{c2}
		\AssignVertexLabel{c}{$1$,$2$,$3$,$4$,$5$}
		\Edges(c0,c1,c2,c0)
		\Edges(c1,c3)
		\Edges(c2,c4)
	\end{tikzpicture}\hspace{4ex}}\qquad
	\subfigure[A standard centered regular $2$-simplex\label{fig:example_b}]{\hspace{8ex}
	\begin{tikzpicture}[line join = round, line cap = round, scale=1, transform shape]
		\pgfmathsetmacro{\factor}{1};
		\pgfmathsetmacro{\arist}{0.3};
		\pgfmathsetmacro{\verti}{0.6};
		\pgfmathsetmacro{\eje}{0.1};
		\def\sep{1.3}
		\coordinate [label=above right:{\color{red!50!black}$u_1$}] (A) at (1 , 0);
		\coordinate [label=above left:{\color{green!30!black}$u_2$}] (B) at (-0.5,  0.866);
		\coordinate [label=below left:{\color{blue!50!black}$u_3$}] (C) at ( -0.5,  -0.866);
		
		\draw[-,line width=\eje mm,gray] (-\sep,0) -- (\sep,0);
		\draw[-,line width=\eje mm,gray] (0,-\sep) -- (0,\sep);
		
		\draw[->,line width=\verti mm,red!80] (0,0)--(A);
		\draw[->,line width=\verti mm,green!70!black] (0,0)--(B);
		\draw[->,line width=\verti mm,blue!80] (0,0)--(C);
		\draw[dashed,line width=\arist mm,black!50,opacity=0.8] (A)--(B)--(C)--cycle;
\end{tikzpicture}\hspace{8ex}
	}
	\caption{Graphical representation of~\Cref{example}}\label{fig:example}
\end{figure}
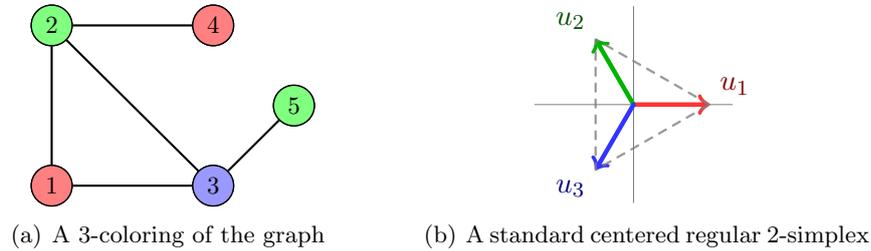
\end{example}

\begin{proposition}\label{prop:equiv}
Let $G=(V,E)$ be a graph with $n$ nodes and let $K$ be a set of $m$ colors. Consider a matrix $X\in\R^{n\times n}$ that verifies properties~\ref{(P1)}--\ref{(P3)}. Then, there exists a proper $m$-coloring $c:V\mapsto K$ such that
$$X=U_c^TU_c,$$
where $U_c$ is given by~\eqref{eq:Uc0}.
\end{proposition}
\begin{proof}
Consider the spectral decomposition  $X=Q\Lambda Q^T,$ where $\Lambda=\diag(\lambda_1,\lambda_2,\ldots,\lambda_n)$ is the diagonal matrix of eigenvalues. Since $X$ is positive definite and has rank not greater than $m-1$, we can assume without loss of generality that $\lambda_1\geq \lambda_2\geq\cdots\geq\lambda_{m-1}\geq 0=\lambda_{m}=\cdots=\lambda_n$. Then, we can express
\begin{equation*}
X=Q\Lambda Q^T=\left(\begin{array}{cc} Q_{11} & Q_{12}\\ Q_{21} & Q_{22}\end{array}\right)\left(\begin{array}{cc} \widehat{\Lambda} & 0\\ 0 & 0\end{array}\right)\left(\begin{array}{cc} Q_{11}^T & Q_{21}^T\\ Q_{12}^T & Q_{22}^T\end{array}\right)=\left(\begin{array}{cc} Q_{11}\\ Q_{21}\end{array}\right)\widehat{\Lambda}\left(\begin{array}{cc} Q_{11}^T & Q_{21}^T\end{array}\right),
\end{equation*}
with $\widehat{\Lambda}=\diag(\lambda_1,\ldots,\lambda_{m-1})$. Hence, we can factorize $X=Y^TY$, with $Y=\widehat{\Lambda}^{\frac{1}{2}}\left(\begin{array}{cc}Q_{11}^T & Q_{21}^T\end{array}\right)$. Let $y_1,\ldots,y_n\in\R^{m-1}$ be the columns of $Y$, i.e., $Y=[y_1|y_2|\cdots|y_n]$. Observe that $y_1,\ldots,y_n$ are unit vectors because $X$ has ones on the diagonal, and thus
\begin{equation}\label{eq:prop_dot_prod}
\langle y_i,y_j \rangle=\left\{\begin{array}{cl} 1, & \text{ if } y_i=y_j;\\ \frac{-1}{m-1}, & \text{ if } y_i\neq y_j;\end{array}\right.
\end{equation}
for all $i,j=1,\ldots,n$. Let us show now that there are at most $m$ distinct vectors among them. To this aim, suppose that $y_{i_1},y_{i_2},\ldots,y_{i_{m+1}}$ are $m+1$ different vectors. Consider
\begin{equation*}\def\arraystretch{1.2}
\widetilde{X}:=\left[\begin{array}{c} y_{i_1}^T\\ y_{i_2}^T\\ \vdots\\ y_{i_{m+1}}^T\end{array}\right]\left[y_{i_1} | y_{i_2} | \cdots | y_{i_{m+1}}\right]=\left(\begin{array}{cccc} 1 & \frac{-1}{m-1} & \cdots & \frac{-1}{m-1}\\
\frac{-1}{m-1}  & 1 & \cdots & \frac{-1}{m-1}\\
\vdots & \vdots & \ddots & \vdots\\
\frac{-1}{m-1} & \frac{-1}{m-1} & \cdots & 1\end{array}\right)\in\R^{(m+1)\times(m+1)}.
\end{equation*}
It holds that $\rank(\widetilde{X})\leq\rank(X)\leq m-1$, since $\widetilde X$ is a submatrix of $X$, but this is a contradiction with the fact that
\begin{equation*}
\det(\widetilde{X})=\left(1-\frac{m}{m-1}\right)\left(1+\frac{1}{m-1}\right)^m\neq 0.
\end{equation*}
Therefore, it must hold that $\cup_{j=1}^n\{y_j\}=\{u_1,\ldots,u_r\}$, where $u_1,\ldots,u_r$ are $r\leq m$ distinct vertices of a regular $(m-1)$-simplex (a rotation of the standard simplex). Finally, define $c:V\mapsto K$ by $c(i)=\left\{k\in\{1,\ldots,r\}\,:\, y_i=u_k\right\}$, so that we trivially get $Y=U_c$, where $U_c$ is as in \eqref{eq:Uc0}. According to~\ref{(P3)}, together with~\eqref{eq:prop_dot_prod}, we have that $c$ is a proper $m$-coloring of $G$, as claimed.
\end{proof}

In view of~\Cref{prop:equiv}, finding a proper $m$-coloring of a graph with $n$ vertices is equivalent to finding an $n\times n$ matrix verifying properties~\ref{(P1)},~\ref{(P2)} and~\ref{(P3)}. In this work, the latter will be tackled by solving the following feasibility problem:

\begin{equation}\label{eq:formulation}
   \text{Find }X\in C_1\cap C_2 \subseteq \R^{n\times n},
\end{equation}
where the constraint sets are defined by
\begin{subequations}\label{eq:constraints}
 	\begin{align}
 	C_1  &:= \left\{
 	X\in\left\{1,\frac{-1}{m-1}\right\}^{n\times n}: x_{ii}=1, \forall i\in  V \text{ and }  x_{ij}=\frac{-1}{m-1}, \forall \{i,j\}\in E
 	\right\}, \label{eq:C1} \\
 	C_2  &:= \left\{ X\in\sdp^n: \rank(X)\leq m-1 \right\}. \label{eq:C2}
 	\end{align}
\end{subequations}

\begin{remark}\label{rmk:equiv_sol}
One advantage of the feasibility problem~\eqref{eq:formulation} is the avoidance of equivalent colorings in the following sense. Suppose that $c:V\mapsto K$ is a proper $m$-coloring of a graph~$G$, and let $W_c$ be its associated matrix given by~\eqref{eq:Uc}. For any permutation of the colors, $\sigma:K\mapsto K$, we have that $\sigma\circ c$ is also a proper $m$-coloring of $G$, so there exist many equivalent valid colorings. However, observe that $W_{\sigma\circ c}=W_c$, and thus all of them lead to a unique solution of~\eqref{eq:formulation}.
\end{remark}

\subsection{Modeling precoloring problems}\label{sec:precol_model}

The \emph{precoloring problem} consists in obtaining a proper coloring of a graph in which the color of some nodes are predefined. Let $\widetilde{V} \subseteq V$ be the subset of precolored nodes and denote by $p_i\in K$ the preassigned color to node $i\in\widetilde{V}$. The task then is to find a coloring $c:V\mapsto K$ such that
\begin{equation}\label{eq:precol}
c(i)\neq c(j), \text{ for all } \{i,j\}\in E \quad and \quad c(i)=p_i, \text{ for all } i\in\widetilde{V}.
\end{equation}
Notice that any coloring satisfying~\eqref{eq:precol} also verifies
\begin{equation}\label{eq:precol_refor}
c(i)\neq c(j), \text{ for all } \{i,j\}\in E \quad and \quad c(i)=c(j) \Leftrightarrow p_i=p_j,\text{ for all } i,j\in\widetilde{V}.
\end{equation}
In fact, both conditions can be shown to be equivalent in the following sense. Suppose that $c:V\mapsto K$ is a coloring verifying~\eqref{eq:precol_refor}. Then, for any permutation of the colors $\sigma:K\mapsto K$ such that $\sigma(c(i))=p_i$ for all $i\in\widetilde{V}$, one can easily check that $\sigma\circ c$ is a proper coloring for which~\eqref{eq:precol} holds.

Therefore, we shall focus on finding colorings fulfilling condition~\eqref{eq:precol_refor}. The matrix~$W_c$ constructed from $c$ as in~\eqref{eq:Uc}, shall verify now~\ref{(P1)},~\ref{(P2)} and
\begin{itemize}
\item[\mylabel{(P3m)}{({P3'})}] $W_c\in\left\{1,\frac{-1}{m-1}\right\}^{n\times n}$ and some of the entries of $W_c=[w_{ij}]$ are determined as follows:
\begin{gather*}
w_{ij}=1,\ \forall \{i,j\}\in \widehat{I} \quad\text{and}\quad w_{ij}=\frac{-1}{m-1},\ \forall \{i,j\}\in \widehat{E};
\end{gather*}
where $\widehat{I}:=\left\{\{i,i\}:i\in V\right\}\cup\left\{\{i,j\}\subseteq\widetilde{V}: p_i=p_j \right\}$ and $\widehat{E}:=E\cup\left\{ \{i,j\}\subseteq\widetilde{V}: p_i\neq p_j\right\}$.
\end{itemize}
The new modified property \ref{(P3m)} can be incorporated into the formulation of the feasibility problem~\eqref{eq:formulation} by replacing the constraint $C_1$ by
\begin{equation}\label{eq:C1precol}
\widehat{C}_1 := \left\{
 	X\in\left\{1,\frac{-1}{m-1}\right\}^{n\times n}: x_{ij}=1, \forall \{i,j\}\in  \widehat{I} \text{ and }  x_{ij}=\frac{-1}{m-1}, \forall \{i,j\}\in \widehat{E}
 	\right\}.
\end{equation}

\begin{example}[\Cref{example} revisited]\label{example2}
Consider the graph in~\Cref{example} and suppose that node $2$ is precolored red (R), and nodes $4$ and $5$ are precolored blue (B). The precoloring problem is shown in~\Cref{fig:example2_a}. Following the notation established above, we have
$$\widehat{I}=\left\{\{i,i\}: i\in V\right\}\cup \left\{\{4,5\}\right\} \quad \text{and}\quad \widehat{E}=E\cup \left\{\{2,4\},\{2,5\}\right\}=E\cup \left\{\{2,5\}\right\}.$$
The unique solution to the feasibility problem $\widehat{C}_1\cap C_2$ is the matrix
{\def\arraystretch{1.7}
\begin{equation*}
W_c=\left(\begin{array}{ccccc} \boxed{1} & \boxed{-0.5} & \boxed{-0.5} & 1 & 1\\
\boxed{-0.5} & \boxed{1} & \boxed{-0.5} & \boxed{-0.5} & \boxed{\boxed{-0.5}}\\
\boxed{-0.5} & \boxed{-0.5} & \boxed{1} & -0.5 & \boxed{-0.5}\\
1 & \boxed{-0.5} & -0.5 & \boxed{1} & \boxed{\boxed{1}}\\
1 & \boxed{\boxed{-0.5}} & \boxed{-0.5} & \boxed{\boxed{1}} & \boxed{1} \end{array}\right),
\end{equation*}}%
where the boxed entries in $W_c$ correspond to those determined by~\ref{(P3m)}. The entries whose values are fixed by~\ref{(P3m)} but not by~\ref{(P3)} are marked with a double-box.

Suppose that we obtain the factorization $W_c=U_c^TU_c$, with $U_c=\left[u_1,u_2,u_3,u_1,u_1\right]$. The $3$-coloring determined by $U_c$ is represented in~\Cref{fig:example2_b}. Then, in order to make this coloring consistent with the precoloring of the vertices, we need to  suitably permute the set of colors. Precisely, we require $\sigma(\text{G})=\text{R}$ and $\sigma(\text{R})=\text{B}$. It must therefore be $\sigma(\text{B})=\text{G}$. The permuted $3$-coloring consistent with the precoloring, given by ${U}_{\sigma\circ c}=[u_3,u_1,u_2,u_3,u_3]$, is shown in~\Cref{fig:example2_c}.

\begin{figure}[ht!]
	\centering
	\subfigure[Precolored graph\label{fig:example2_a}]{\hspace{1ex}
	\begin{tikzpicture}[scale=.85,transform shape]%
		\def\sep{1.25}
		\GraphInit[vstyle=Normal]
		\SetVertexNoLabel
		\tikzstyle{EdgeStyle}=[thick]
		\coordinate (A) at (-5,0);
		\coordinate (B) at (5,0);
		\Vertex[x=-\sep, y=-\sep] {c0}
		\Vertex[x=-\sep, y=\sep] {c1}
		\Vertex[x=\sep, y=-\sep] {c2}
		\Vertex[x=\sep, y=\sep] {c3}
		\Vertex[x=2*\sep, y=0] {c4}
		\AddVertexColor{blue!40}{c4,c3}
		\AddVertexColor{red!50}{c1}
		\AddVertexColor{green!50}{}
		\AssignVertexLabel{c}{$1$,$2$,$3$,$4$,$5$}
		\Edges(c0,c1,c2,c0)
		\Edges(c1,c3)
		\Edges(c2,c4)
	\end{tikzpicture}\hspace{1ex}}\qquad
	\subfigure[A $3$-coloring of the graph\label{fig:example2_b}]{\hspace{1ex}
	\begin{tikzpicture}[scale=.85,transform shape]%
		\def\sep{1.25}
		\GraphInit[vstyle=Normal]
		\SetVertexNoLabel
		\tikzstyle{EdgeStyle}=[thick]
		\coordinate (A) at (-5,0);
		\coordinate (B) at (5,0);
		\Vertex[x=-\sep, y=-\sep] {c0}
		\Vertex[x=-\sep, y=\sep] {c1}
		\Vertex[x=\sep, y=-\sep] {c2}
		\Vertex[x=\sep, y=\sep] {c3}
		\Vertex[x=2*\sep, y=0] {c4}
		\AddVertexColor{green!50}{c1}
		\AddVertexColor{red!50}{c0,c3,c4}
		\AddVertexColor{blue!40}{c2}
		\AssignVertexLabel{c}{$1$,$2$,$3$,$4$,$5$}
		\Edges(c0,c1,c2,c0)
		\Edges(c1,c3)
		\Edges(c2,c4)
	\end{tikzpicture}\hspace{1ex}}\qquad\subfigure[A permutation consistent with the precoloring\label{fig:example2_c}]{\hspace{1ex}
	\begin{tikzpicture}[scale=.85,transform shape]%
		\def\sep{1.25}
		\GraphInit[vstyle=Normal]
		\SetVertexNoLabel
		\tikzstyle{EdgeStyle}=[thick]
		\coordinate (A) at (-5,0);
		\coordinate (B) at (5,0);
		\Vertex[x=-\sep, y=-\sep] {c0}
		\Vertex[x=-\sep, y=\sep] {c1}
		\Vertex[x=\sep, y=-\sep] {c2}
		\Vertex[x=\sep, y=\sep] {c3}
		\Vertex[x=2*\sep, y=0] {c4}
		\AddVertexColor{green!50}{c2}
		\AddVertexColor{blue!40}{c3,c4,c0}
		\AddVertexColor{red!50}{c1}
		\AssignVertexLabel{c}{$1$,$2$,$3$,$4$,$5$}
		\Edges(c0,c1,c2,c0)
		\Edges(c1,c3)
		\Edges(c2,c4)
	\end{tikzpicture}\hspace{1ex}}
	\caption{Graphical representation of~\Cref{example2}}\label{fig:example2}
\end{figure}
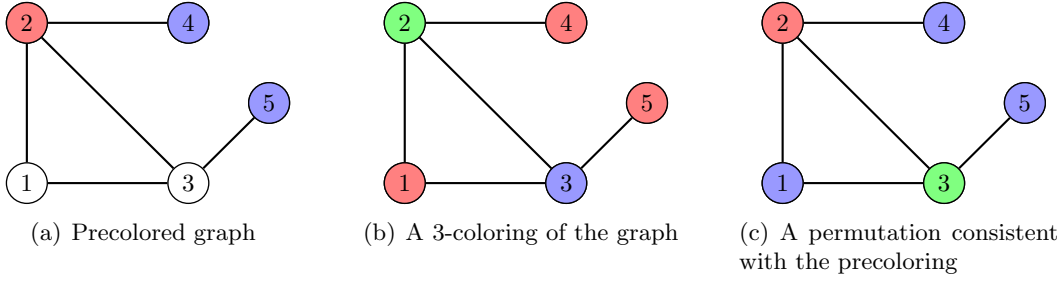
\end{example}

\section{Implementation of the Douglas--Rachford algorithm}\label{sec:implementation}

In order to apply the DR algorithm to feasibility problems, and \eqref{eq:formulation} in particular, it must be possible to efficiently compute the projections onto the two constraint sets, in our case~\eqref{eq:constraints}. This is indeed the case, as shown in the following results.

\begin{proposition}[Projection onto $C_1$]\label{prop:PC1} Consider any $X=(x_{ij})\in\R^{n\times n}$. A projection of $X$ onto the set $C_1$ defined in~\eqref{eq:C1} is given componentwise by
\begin{equation}\label{eq:PC1}
\Big(\pi_{C_1}(X)\Big)[i,j]=\left\{\begin{array}{ll} 1,&\text{if } x_{ij}>\frac{m-2}{2(m-1)} \text{ and } \{i,j\}\not\in E,  \text{ or } i=j;\\
\frac{-1}{m-1} &,\text{if } x_{ij}\leq\frac{m-2}{2(m-1)}  \text{ and } i\neq j, \text{ or } \{i,j\}\in E.\end{array}\right.
\end{equation}
A projection of $X$ onto the set $\widehat{C}_1$ in~\eqref{eq:C1precol} is given componentwise by
\begin{equation}\label{eq:PC1hat}
\Big(\pi_{\widehat{C}_1}(X)\Big)[i,j]=\left\{\begin{array}{ll} 1,&\text{if } x_{ij}>\frac{m-2}{2(m-1)} \text{ and } \{i,j\}\not\in \widehat{E},  \text{ or } \{i,j\}\in \widehat{I};\\
\frac{-1}{m-1}, &\text{if } x_{ij}\leq\frac{m-2}{2(m-1)}  \text{ and } \{i,j\}\not\in \widehat{I}, \text{ or } \{i,j\}\in \widehat{E}.\end{array}\right.
\end{equation}
\end{proposition}
\begin{proof}
Clearly, the projector of $X$ onto $C_1$ can be computed componentwise. Taking into account the constraints in~\eqref{eq:C1}, the projection of an entry $x_{ij}$ is $1$ if $i=j$, and is $\frac{-1}{m-1}$ if $\{i,j\}\in E$. Otherwise, it is equal to $P_{\left\{1,\frac{-1}{m-1}\right\}}(x_{ij})$. As the middle point between these two values is $\frac{m-2}{2(m-1)}$, then~\eqref{eq:PC1} follows. The proof of~\eqref{eq:PC1hat} is analogous.
\end{proof}

\begin{proposition}[Projection onto $C_2$]\label{prop:PC2} Let $X\in\Sy^n$ and consider its spectral decomposition $X=Q\Lambda Q^T$, with $\Lambda=\diag(\lambda_1,\ldots,\lambda_n)$ and $\lambda_1\geq\lambda_2\geq\cdots\geq\lambda_n$. A projection of $X$ onto the set $C_2$ defined in~\eqref{eq:C2} is given componentwise by
\begin{equation}\label{eq:PC2}
\pi_{C_2}(X)=Q\Lambda^{+}_{m-1}Q^T,
\end{equation}
where $\Lambda^{+}_{m-1}=\diag\left(\max\{0,\lambda_1\},\ldots,\max\{0,\lambda_{m-1}\},0,\ldots,0 \right)$.
\end{proposition}
\begin{proof}
See, e.g., \cite[Proposition~3.11]{T17}.
\end{proof}

\begin{remark}
According to \Cref{prop:PC1,prop:PC2}, computing a projection onto $C_1$ is a simple rounding operation, while a projection onto $C_2$ requires the computation of the spectral decomposition of an $n\times n$ matrix. From a computational point of view, the former is not a problem but the later may be time-consuming, especially for big problems. However, observe that we do not need to compute the whole spectrum in~\eqref{eq:PC2}, but only the $m-1$ largest eigenvalues and their associated eigenvectors. In large-scale problems, $m$ is usually much smaller than~$n$ and hence $\pi_{C_2}$ can be computed reasonably fast.

Constraint non-convexity manifests itself in the equality case of the conditionals in \Cref{prop:PC1}, and the case of degenerate eigenvalues in \Cref{prop:PC2}. Neither of these can be acted upon in practice, given the finite precision of the computations.
\end{remark}

\begin{remark}\label{rem:sym}
In order to find $\pi_{C_2}(X)$, \Cref{prop:PC2} requires the matrix $X$ to be symmetric. Observe that, according to~\eqref{eq:PC1} and by definition of $C_2$, we get that
\begin{equation*}
\pi_{C_1}(X), \pi_{C_2}(X)\in\Sy^n, \quad \text{ for all } X\in\Sy^n.
\end{equation*}
Hence, since $\Sy^n$ is a subspace, the iterates generated by DR~\eqref{eq:NonconvexDR} will remain symmetric (with due attention to numerical precision), as long as the initial point is chosen in $\Sy^n$.
\end{remark}

There are several options for implementing the DR algorithm. The simplest choice would be to directly apply DR in the original
space $\R^{n\times n}$, since the feasibility problem to be solved~\eqref{eq:formulation} only involves two constraint sets.
Then, we can iterate by using either $T_{C_1,C_2,\lambda}$ or $T_{C_2,C_1,\lambda}$.
On the other hand, although the product space reformulation is typically employed for feasibility problems involving more than two sets, it can still be applied to two sets.
In this way, we obtain two additional implementations by either using the operator $T_{D,C,\lambda}$ or $T_{C,D,\lambda}$.
The purpose of the next section is to numerically compare these different implementations. In our numerical tests we observed that the numerical behavior of $T_{D,C,\lambda}$ and $T_{C,D,\lambda}$
is similar; thus, to simplify, we only show the results for the operator $T_{D,C,\lambda}$, whose \emph{shadow} $P_D\circ T_{D,C,\lambda}$ is easier to track, as it can be identified with a sequence in the original space $\R^{n\times n}$.

\section{Numerical experiments}\label{sec:numexp}

In this section we run various numerical experiments to test the performance of the DR algorithm for solving different graph coloring problems. We compare the formulation discussed in Section~\ref{sec:formulation} with the one recently proposed in~\cite{AC18}. To distinguish them, we shall refer to the model proposed in~\cite{AC18} as the \emph{binary formulation}, and to the new one developed in~\Cref{sec:formulation} as the \emph{rank formulation}.

In each of the next five subsections we run illustrative experiments on different families of graphs: the Queens\_$n^2$ puzzles, random colorable graphs, the windmill graphs, Sudokus, and the DIMACS benchmark instances. Each of these families is employed for a different purpose. We start with a difficult coloring problem, the Queens\_$n^2$ puzzle, where we show the effect that the parameter $\lambda$ has in the different implementations. We also use these puzzles to draw attention to something that is usually overlooked: finite machine precision. Next, to test how the method scales, we run an experiment on random colorable graphs with controlled asymptotic complexity. The windmill graphs and the Sudoku puzzles are used to show that the rank formulation is superior to the binary formulation, even when we allow maximal clique information. We finish this experimental section by testing the algorithm on the DIMACS benchmark instances, a widely used collection of diverse graph types.

Unless otherwise stated, the stopping criterion used for the implementation $T_{A,B,\lambda}$ was
\begin{equation}\label{eq:stop}
{\rm Error}_{A,B}(x_k):=\left\|P_B(P_A(x_k))-P_A(x_k)\right\|\leq 10^{-10},
\end{equation}
where $x_k$ is the current iterate, in which case the solution attempt was labeled as successful. All codes were written in Python~2.7 and the tests were run on an Intel Core i7-4770 CPU \@3.40GHz with 32GB RAM, under Windows~10 (64-bit).

\subsection{Queens\_$n^2$ puzzles}

A well-known and challenging graph coloring problem is the Queens\_$n^2$ puzzle.
This puzzle consists in covering the entire $n\times n$ chessboard with queens of different colors, so that two queens of the same color do not attack each other.
The puzzle is equivalent to finding a proper coloring of a particular graph, which has a vertex at each cell of the chessboard and edges between all pairs of vertices (cells) that lie on the same column, row or diagonal.
In~\Cref{tbl:nQ_chrom} we give the chromatic number of the graph for the first nine puzzles. The smallest open case for which the chromatic number is currently unknown is $n=27$, see~\cite{OEIS}.

\begin{table}[ht!]
\centering
\begin{tabular}{|c|c|c|c|c|c|c|c|c|c|c|}
\hline
$n$ & 2 & 3 & 4 & 5 & 6 & 7 & 8 & 9 & 10 \\
\hline
$\chi(n)$ & 4 & 5 & 5 & 5 & 7 & 7 & 9 & 10 & 11 \\
\hline
\end{tabular}
\caption{Chromatic number $\chi(n)$ of the Queens\_$n^2$ graph~\cite{OEIS}}
\label{tbl:nQ_chrom}
\end{table}

In our first experiment, we analyze how both the implementation and the choice in the relaxation parameter $\lambda$ affects the behavior of DR for solving this type of puzzle. For each $n\in\{3,4,\ldots,10\}$ and each $\lambda\in\{0.25, 0.5,\ldots,1.75\}$, we ran three different implementations of DR (namely, $T_{C_1,C_2,\lambda}$, $T_{C_2,C_1,\lambda}$ and $T_{D,C,\lambda}$) from $10$ random starting points. The results are shown in~\Cref{fig:nqueens}, where the markers correspond to the median among the solved instances. We also show the percentage of instances solved for each value of $\lambda$, among all the problems and repetitions. According to these results, it seems that the value of the parameter $\lambda$ that suits best each of the formulations $T_{C_1,C_2,\lambda}$, $T_{C_2,C_1,\lambda}$ and $T_{D,C,\lambda}$, is $\lambda=0.75$, $\lambda=0.5$ and $\lambda=1$, respectively. To corroborate this conclusion, we visualize the results using performance profiles, which are constructed as follows (see~\cite{DM02} and the modification proposed in~\cite{ISU16}).

\paragraph{Performance profile}
Let $\Phi$ denote the set of formulations, and let $\mathcal{P}$ be a set of $N$ problems. For each formulation $f\in\Phi$, let $t_{f,p}$ be the averaged time required by DR to solve problem $p\in\mathcal{P}$ among all the successful runs, and let $s_{f,p}$ denote the fraction of successful runs for problem $p$. Compute $t^\star_p:=\min_{f\in\Phi} t_{f,p}$ for all $p\in\mathcal{P}$. Then, for any $\tau\geq 1$, define $R_f(\tau):=\{p\in\mathcal{P}, t_{f,p}\leq\tau t^\star_p\}$; i.e., $R_f(\tau)$ is the set of problems for which formulation $f$ is at most $\tau$ times slower than the best one. The \emph{performance profile} function of formulation $f$ is given by
\begin{equation*}
\begin{array}{rccl}
\rho_f: & [1,+\infty) & \longmapsto & [0,1]\\
& \tau &  \mapsto & \rho_f(\tau):=\frac{1}{N}\sum_{p\in R_f(\tau)} s_{f,p}.
\end{array}
\end{equation*}
The value $\rho_f(1)$ indicates the portion of runs for which $f$ was the fastest formulation. When $\tau\rightarrow+\infty$, then $\rho_f(\tau)$ gives the fraction of successful runs for formulation $f$.\\

We show in~\Cref{fig:nqueensPP} the performance profiles for each value of $\lambda$ and each of the three implementations $T_{C_1,C_2,\lambda}$, $T_{C_2,C_1,\lambda}$ and $T_{D,C,\lambda}$. This corroborates our previous choice of best parameters $\lambda$ for each implementation. Finally, we compare $T_{C_1,C_2,0.75}$, $T_{C_2,C_1,0.5}$ and $T_{D,C,1}$ in~\Cref{fig:nqueensPP_global}, where we can clearly observe that the first implementation dominates the others.

\begin{figure}[h!]
\addtocounter{subfigure}{-1}
	\centering
	\subfigure{\includegraphics[width=0.67\linewidth]{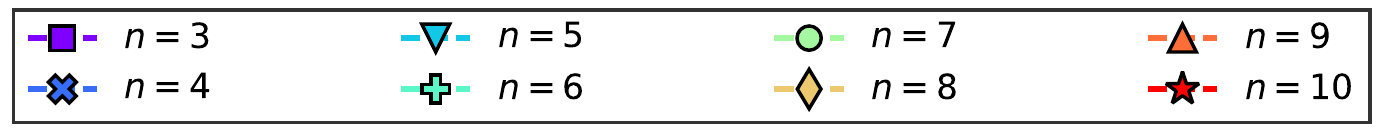}}
	\subfigure[DR implemented with $T_{C_1,C_2,\lambda}$]{\includegraphics[width=0.67\linewidth]{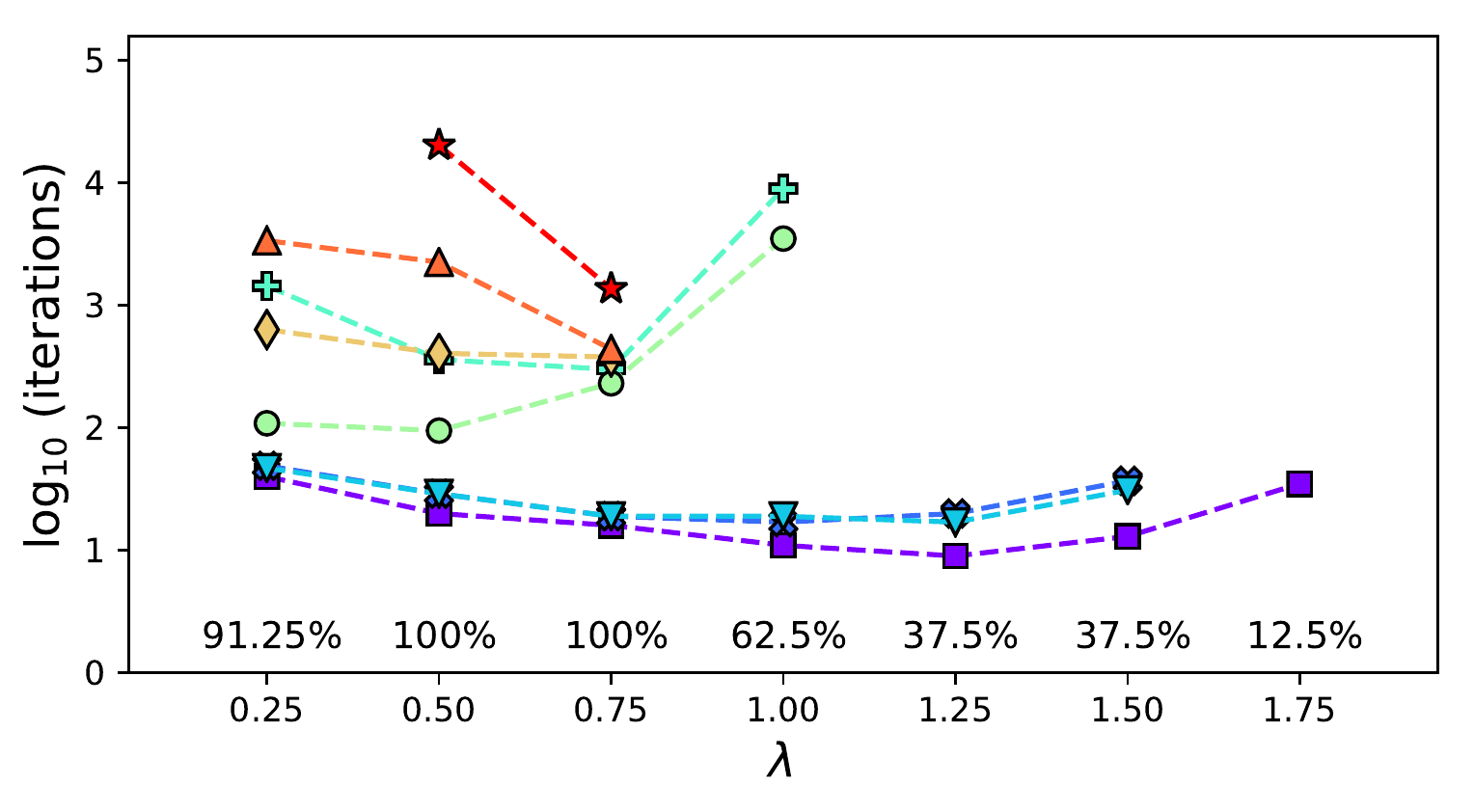}}
	\subfigure[DR implemented with $T_{C_2,C_1,\lambda}$]{\includegraphics[width=0.67\linewidth]{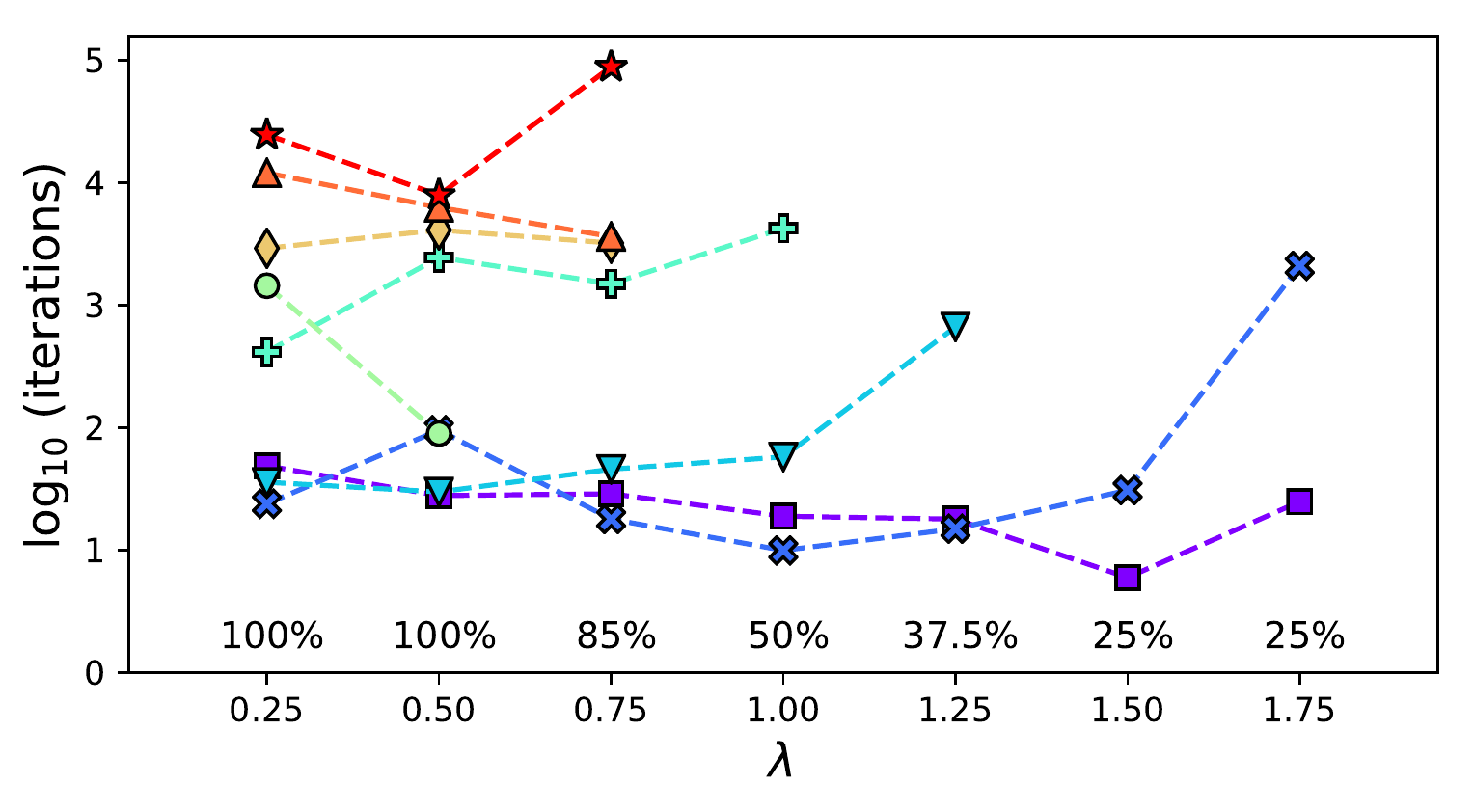}}
	\subfigure[DR implemented with $T_{D,C,\lambda}$]{\includegraphics[width=0.67\linewidth]{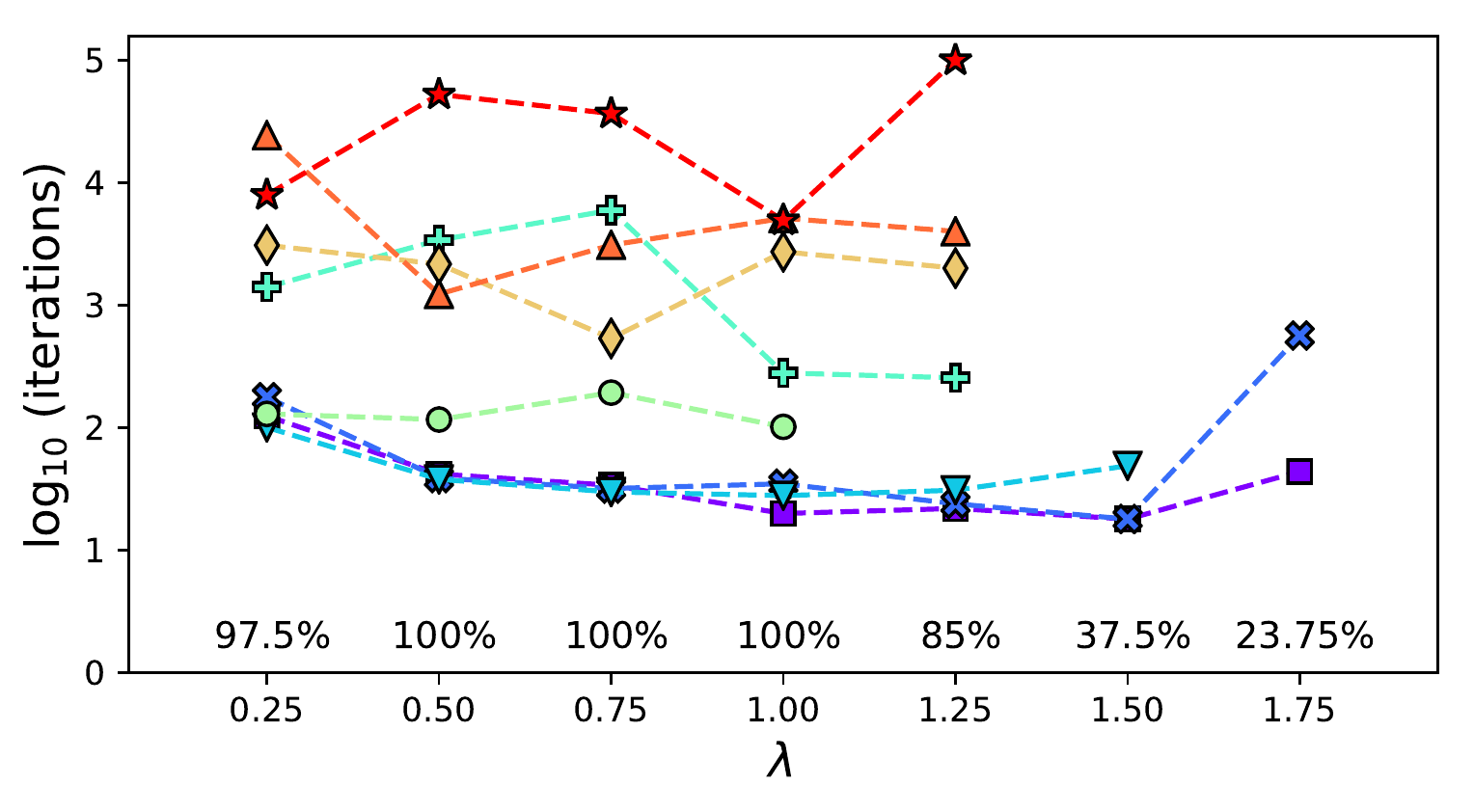}}
	\caption{Results of the Queens\_$n^2$ experiment for three implementations of DR. Each marker corresponds to the median of the solved instances
among $10$ random starting points. At the bottom of each graph, we show the percentage of solved instances for each value of $\lambda$. Instances were considered
as unsolved after 100,000 iterations} \label{fig:nqueens}
\end{figure}

\begin{figure}[h!]
\addtocounter{subfigure}{-1}
	\centering
	\subfigure{\includegraphics[width=0.97\linewidth]{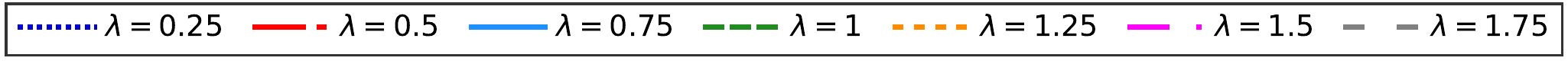}}
	\subfigure[DR implemented with $T_{C_1,C_2,\lambda}$]{\includegraphics[width=0.325\linewidth]{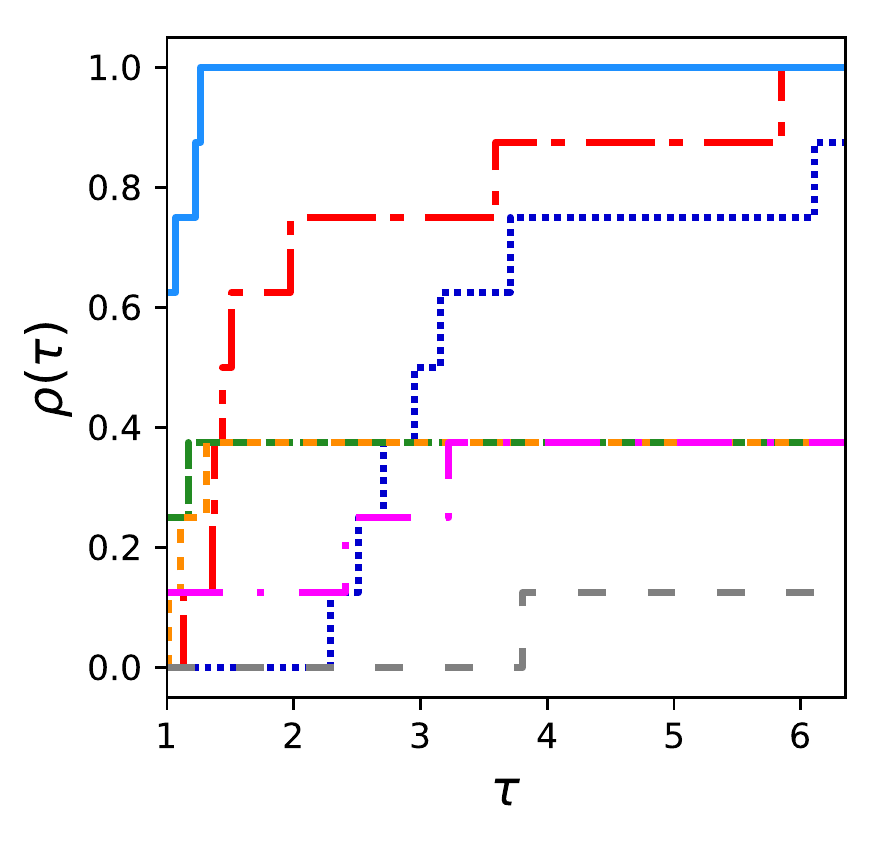}}
	\subfigure[DR implemented with $T_{C_2,C_1,\lambda}$]{\includegraphics[width=0.325\linewidth]{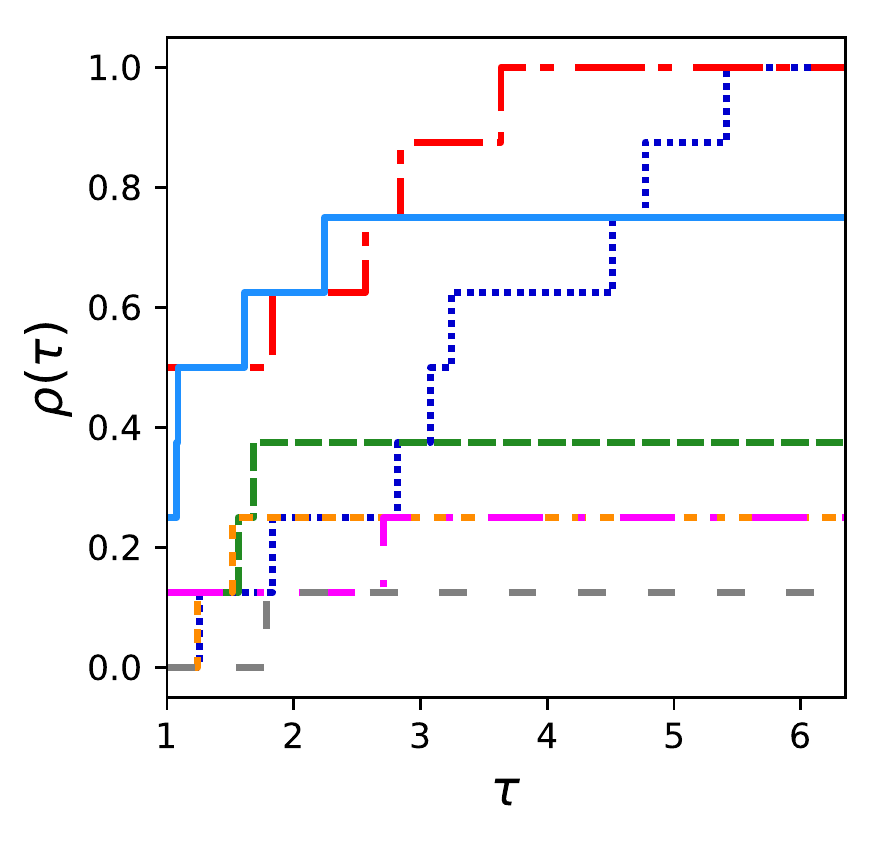}}
	\subfigure[DR implemented with $T_{D,C,\lambda}$]{\includegraphics[width=0.325\linewidth]{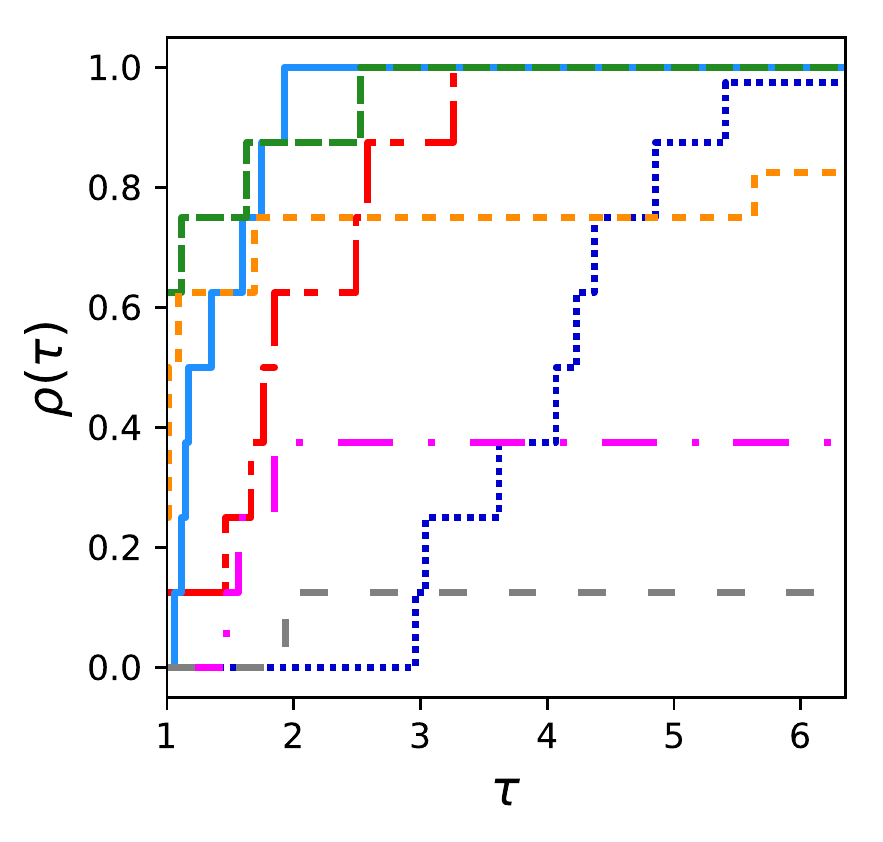}}
	\caption{Performance profiles of the Queens\_$n^2$ experiment for three implementations of DR} \label{fig:nqueensPP}
\end{figure}

\begin{figure}[h!]
\centering
\includegraphics[width=0.65\linewidth]{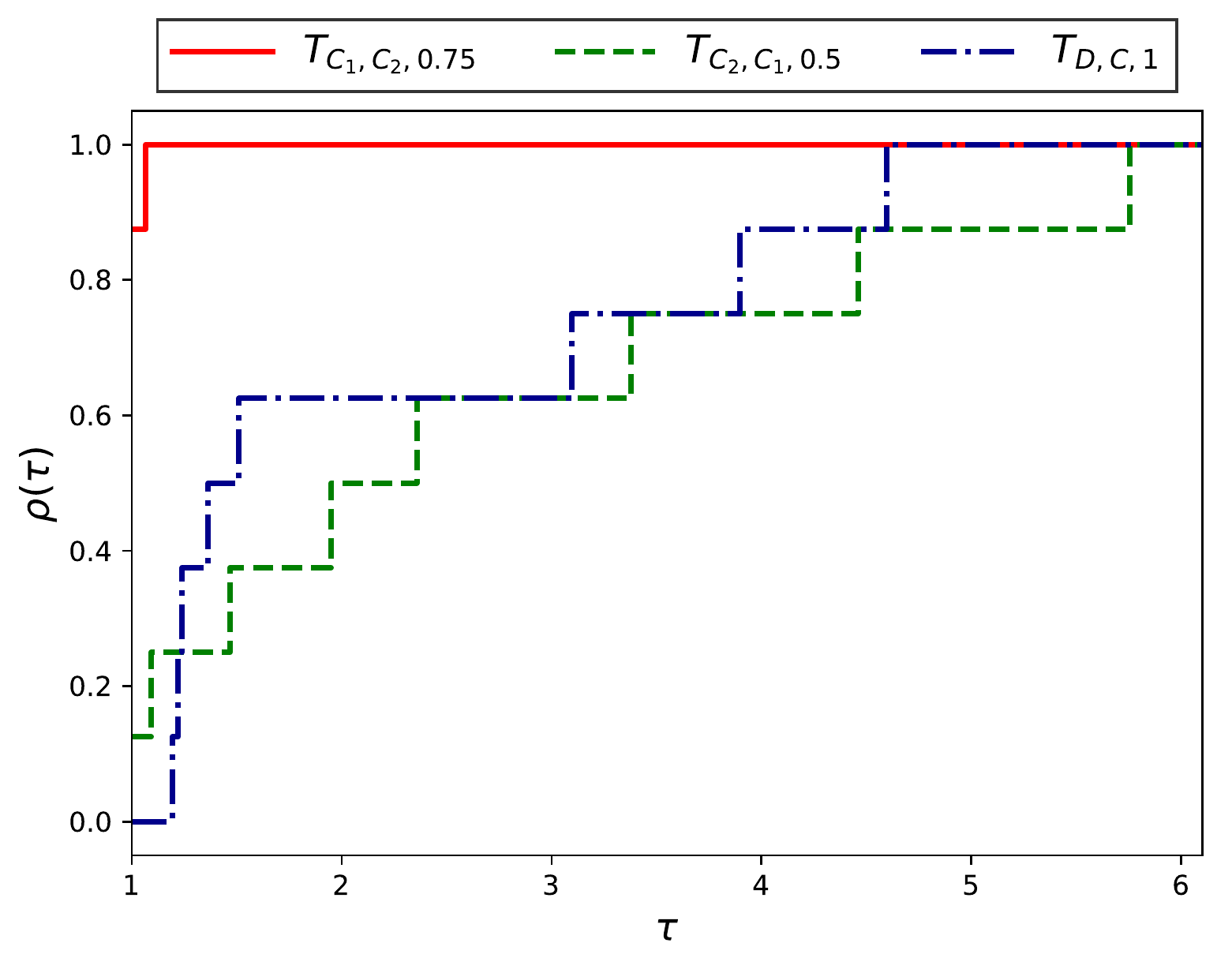}
\caption{Performance profiles of the Queens\_$n^2$ experiment comparing the implementations $T_{C_1,C_2,0.75}$, $T_{C_2,C_1,0.5}$ and $T_{D,C,1}$} \label{fig:nqueensPP_global}
\end{figure}

\begin{remark}[On the machine precision]\label{rem:precision} Our numerical tests show no systematic effect of the machine precision on the average number of iterations per solution, provided the precision is above a modest threshold of about 6 decimal digits. This is consistent with the chaotic dynamics displayed by DR when solving hard problems.
\end{remark}
The behavior explained in Remark~\ref{rem:precision} is demonstrated in the next experiment, where  $T_{D,C,1}$ was implemented for solving the Queens\_$6^2$ and the Queens\_$7^2$ puzzles. For each problem, the algorithm was run from the same starting point using different values of the machine precision. The stopping criterion~\eqref{eq:stop} was decreased to $10^{-5}$ to accommodate the reduced precision. We believe this is still adequate to recover a unique, discrete coloring from the Gram matrix. The results of repeating this experiment for $10$~different random starting points are shown in \Cref{fig:precision}. In~\Cref{fig:digits} we plot the value of ${\rm Error}_{D,C}$ in~\eqref{eq:stop} with respect to the number of iterations for up to 15 digits of precision for one particular random starting point. While these results indicate a high sensitivity to the numerical precision, there is no evidence of a systematic effect. For these experiments, we employed the \texttt{mpmath} library~\cite{mpmath}, which drastically increases the time needed to compute the iterations of the DR algorithm.

\begin{figure}[ht!]
\centering
\includegraphics[width=.7\textwidth]{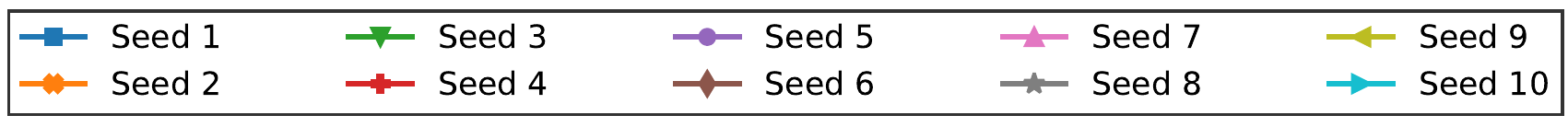}
\subfigure[Queens\_$6^2$]{\includegraphics[width=.495\textwidth]{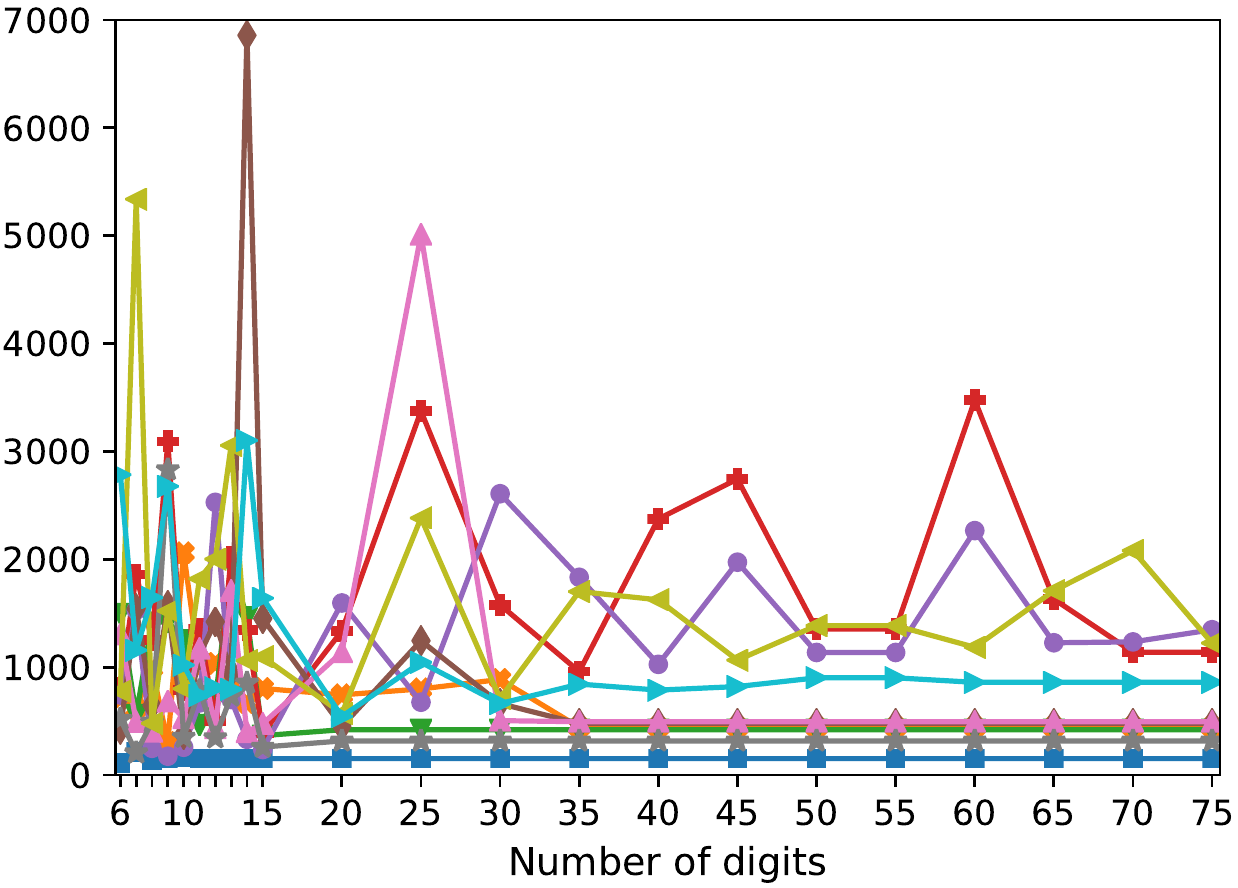}}
\subfigure[Queens\_$7^2$]{\includegraphics[width=.495\textwidth]{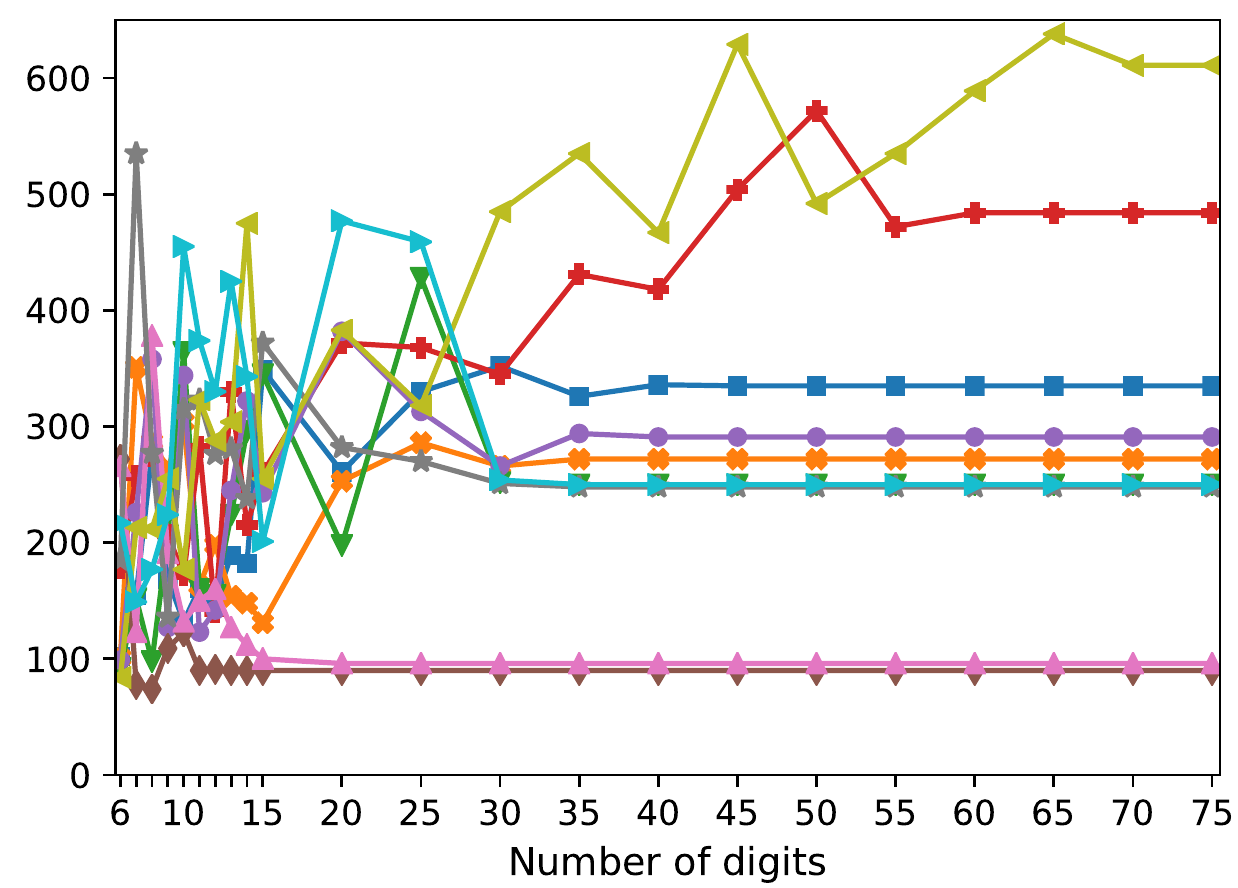}}
\caption{Comparison of the number of number of iterations and the number of digits used in the machine precision for $10$ random starting points, when
  $T_{D,C,1}$ was employed to solve the Queens\_$6^2$ and the Queens\_$7^2$ puzzles. For every starting point and every value of the machine precision, the algorithm found a solution to the puzzle
}\label{fig:precision}
\end{figure}

\begin{figure}[ht!]
\centering
\includegraphics[width=.7\textwidth]{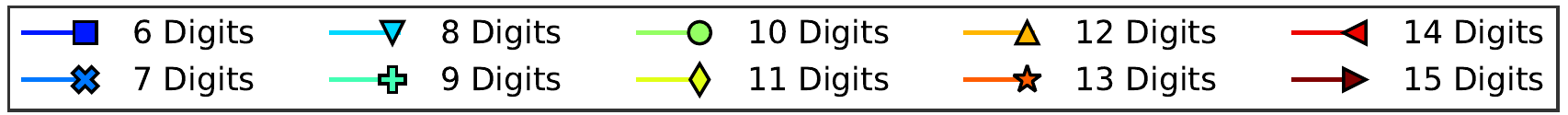}
\subfigure[Queens\_$6^2$]{\includegraphics[width=.495\textwidth]{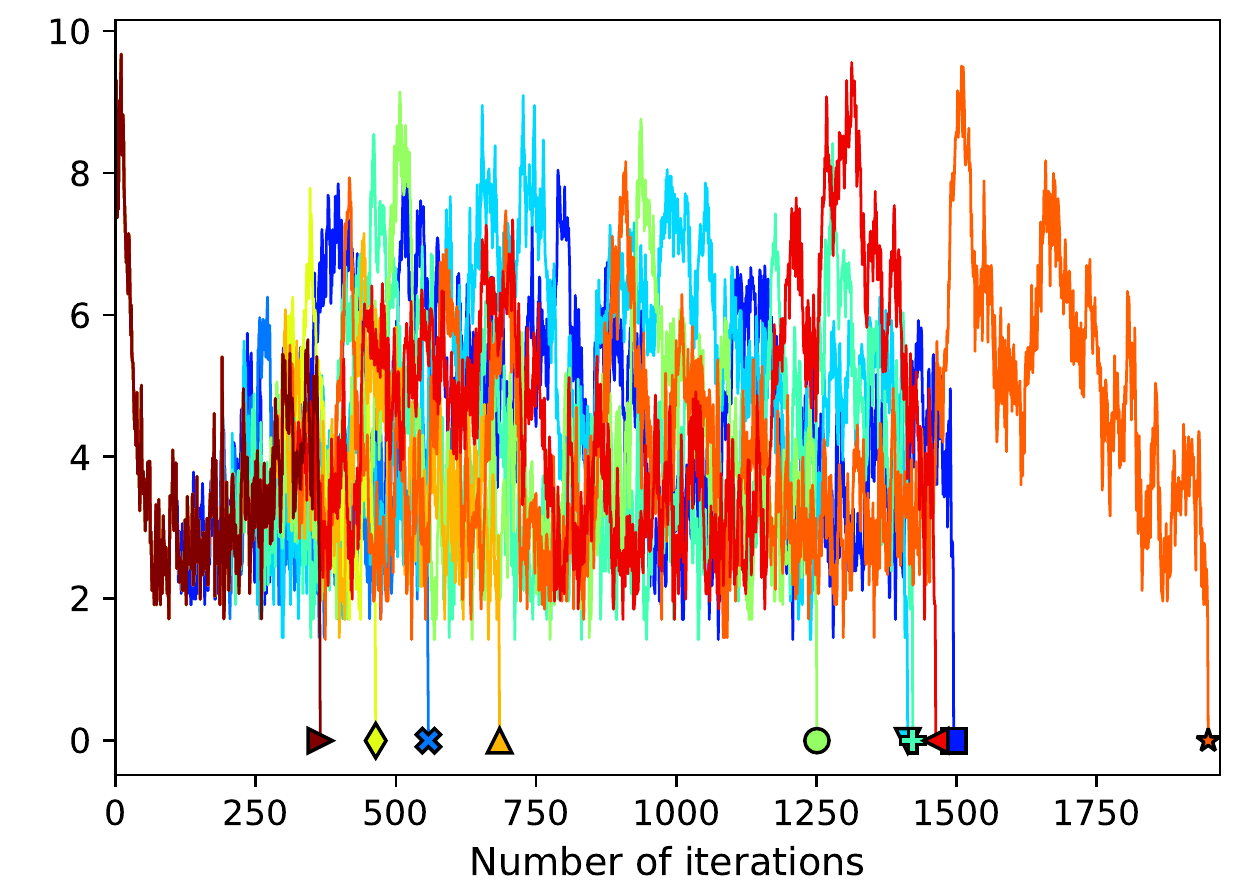}}
\subfigure[Queens\_$7^2$]{\includegraphics[width=.495\textwidth]{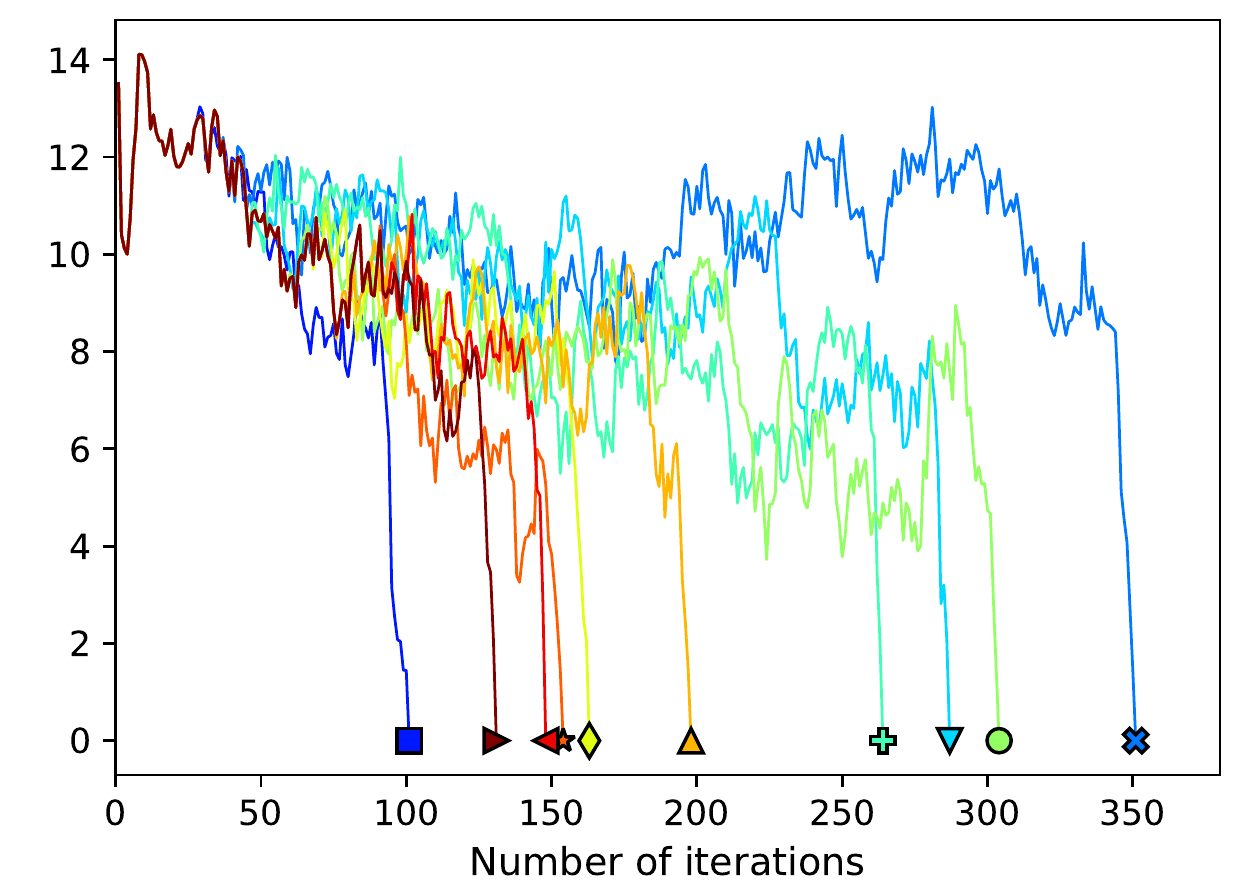}}
\caption{Comparison of the value of ${\rm Error}$ in~\eqref{eq:stop} and the number of iterations for different number of decimal digits used in the machine precision, when
  $T_{D,C,1}$ was employed to solve the Queens\_$6^2$ and the Queens\_$7^2$ puzzles
}\label{fig:digits}
\end{figure}

\subsection{Random colorable graphs}

The hardness of finding a proper coloring of a graph depends on many factors, the single most significant of which is the number of valid colorings. Random colorable graphs are easily constructed, but to be able to draw some consequences from the experiments we run on them, we must generate them in such a way that their complexity is controlled.

We consider the Erd\"{o}s--Renyi model~\cite{ER59}, $\mathbb{G}(\alpha,n)$, which is the ensemble of all graphs with $n$ vertices and $l=\lfloor{\alpha n}\rfloor$ edges, where $\lfloor\cdot\rfloor$ denotes the integer part, endowed with the uniform measure. Hence, $\alpha$ represents the averaged number of edges per node. The probability that a random graph with this distribution is $m$-colorable depends on the magnitude of the parameter~$\alpha$. Precisely, the expected number of proper colorings decreases as $\alpha$ increases. There is an \emph{asymptotic threshold} in the colorable-uncolorable transition denoted $\alpha_s(m)$ (see~\cite[Theorem~1.1]{AF99}). This means that the probability an $m$-coloring exists tends to one as $n$ increases, provided that $\alpha<\alpha_s(m)$, and conversely, it converges to zero for $\alpha>\alpha_s(m)$. The asymptotic threshold is known to be upper-bounded by $\alpha_s(m)\leq \bar{\alpha}_s(m):= \frac{\log m}{\log{\frac{m}{m-1}}}$ (see, e.g.,~\cite[Section~2]{AM99}).

With the number of vertices $n$ and the number of colors $m$ fixed, random graphs sampled from  $\mathbb{G}(\bar{\alpha}_s(m),n)$ are at the $m$-colorability transition and expected to be hard instances, when solvable. In order to avoid non-colorable graphs, the sampling can be modified to ensure the existence of a coloring as follows. First, a partition of $V$ into $m$ groups with approximately equal size is chosen, e.g. consider the equivalence classes defined by the congruence modulo $m$ of the integer vertex labels. Then, $ \lfloor{\bar{\alpha}_s(m) n}\rfloor$ edges are randomly generated from the uniform distribution over the set of all edges connecting two nodes in different groups. \Cref{alg:randG} contains the discussed routine that generates such graphs.

\begin{algorithm}[ht!]
	\caption{Generate an $m$-colorable random graph with low expected number of $m$-colorings}\label{alg:randG}
	\KwIn{$V=\{1,\ldots,n\}$, $m\geq 2$}
	Set $\bar{\alpha}_s(m):=\frac{\log m}{\log{\frac{m}{m-1}}}$, $E:=\emptyset$ and	 $l=0$\;
	\While{{$l<\bar{\alpha}_s(m)n$}}{
		Generate randomly $e:=\{i,j\}\in V\times V$\;
		\If{{$e\not\in E$} \rm{and} {$(i-j)\not\equiv 0 \text{ (mod }m)$}}{
		$E=E\cup \{e\}$\;
		$l=l+1$\;
		}
	}
	\KwOut{$G=(V,E)$}
\end{algorithm}

The goal of our next experiment is to show how the DR algorithm complexity grows with respect to the number of vertices in the graph. We make use of colorable random graphs with low expected number of colors so that we have control of the complexity of our instances.  For each $m\in\{8,9,10\}$ and for each $n\in\{50,75,\cdots,200\}$, we generated $5$ random graphs using~\Cref{alg:randG}. Then, for each graph, the DR algorithm was run from $5$ different starting points (this makes a total of $25$ runs per each pair $(m,n)$). Based on the results in the Queens\_$n^2$ experiment, we implemented DR with $T_{C_1,C_2,\lambda}$. To confirm our previous choice of the best parameter $\lambda=0.75$, we repeated the experiment for each $\lambda\in\{0.25, 0.5,\ldots,1.75\}$. The results shown in~\Cref{fig:rand} confirm that the best choice for general purposes is $\lambda=0.75$. In~\Cref{fig:rand_scale} we plot the number of iterations needed by $T_{C_1,C_2,0.75}$ with respect to the size of the graph, for each $m$. We observe that the number of colors does not have a noticeable effect on the performance of the algorithm. As expected, we come upon an exponential dependence between size and iterations, which is consistent with the NP-hardness of the problem.

\begin{figure}[ht!]
\addtocounter{subfigure}{-1}
	\centering
	\subfigure{\includegraphics[width=0.94\linewidth]{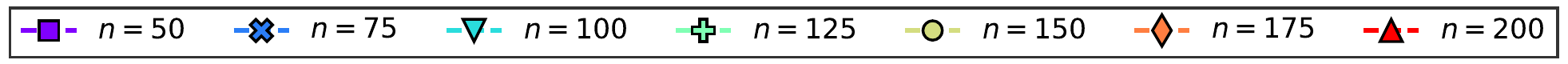}}
	\subfigure[$8$ colors]{\includegraphics[width=0.67\linewidth]{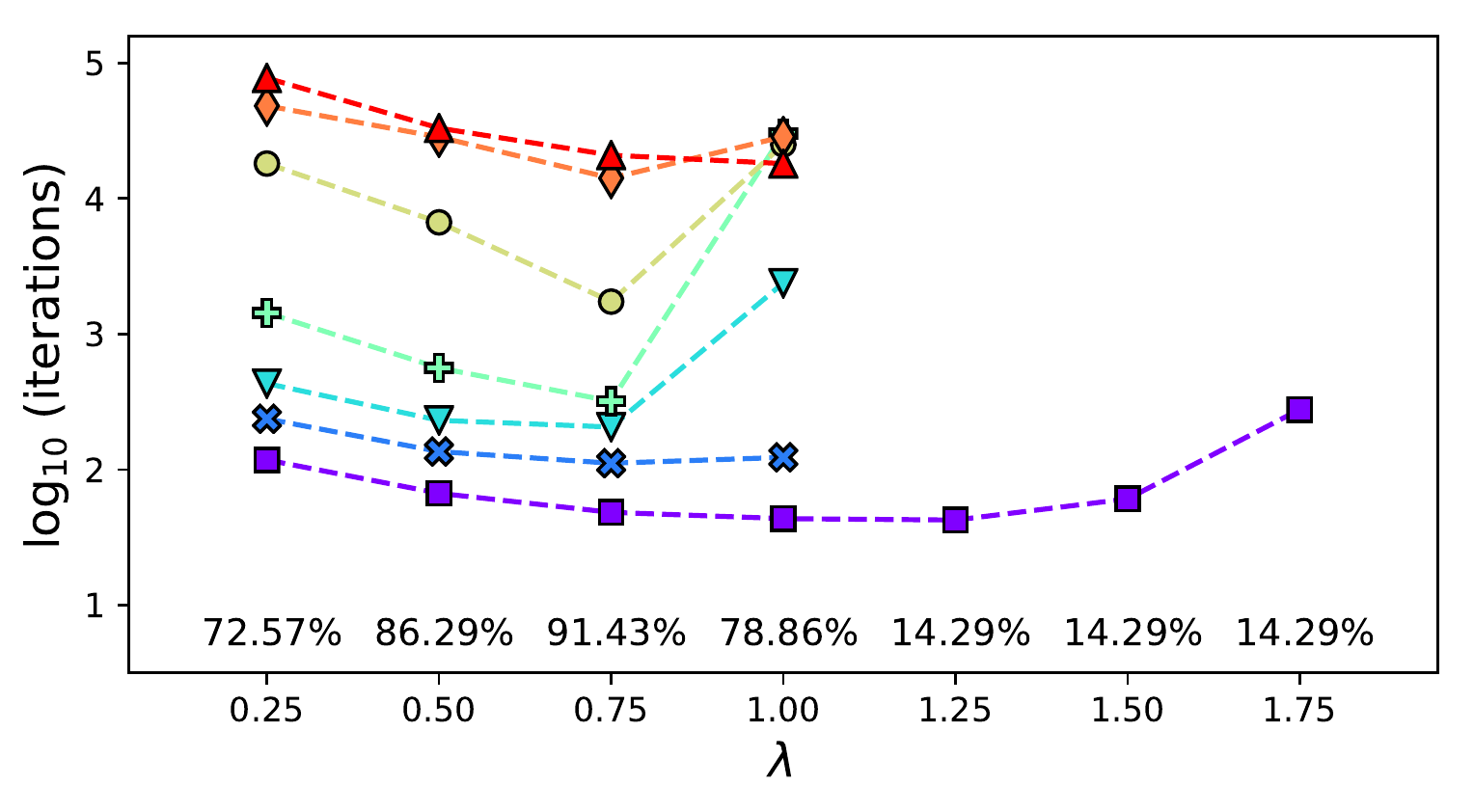}}
	\subfigure[$9$ colors]{\includegraphics[width=0.67\linewidth]{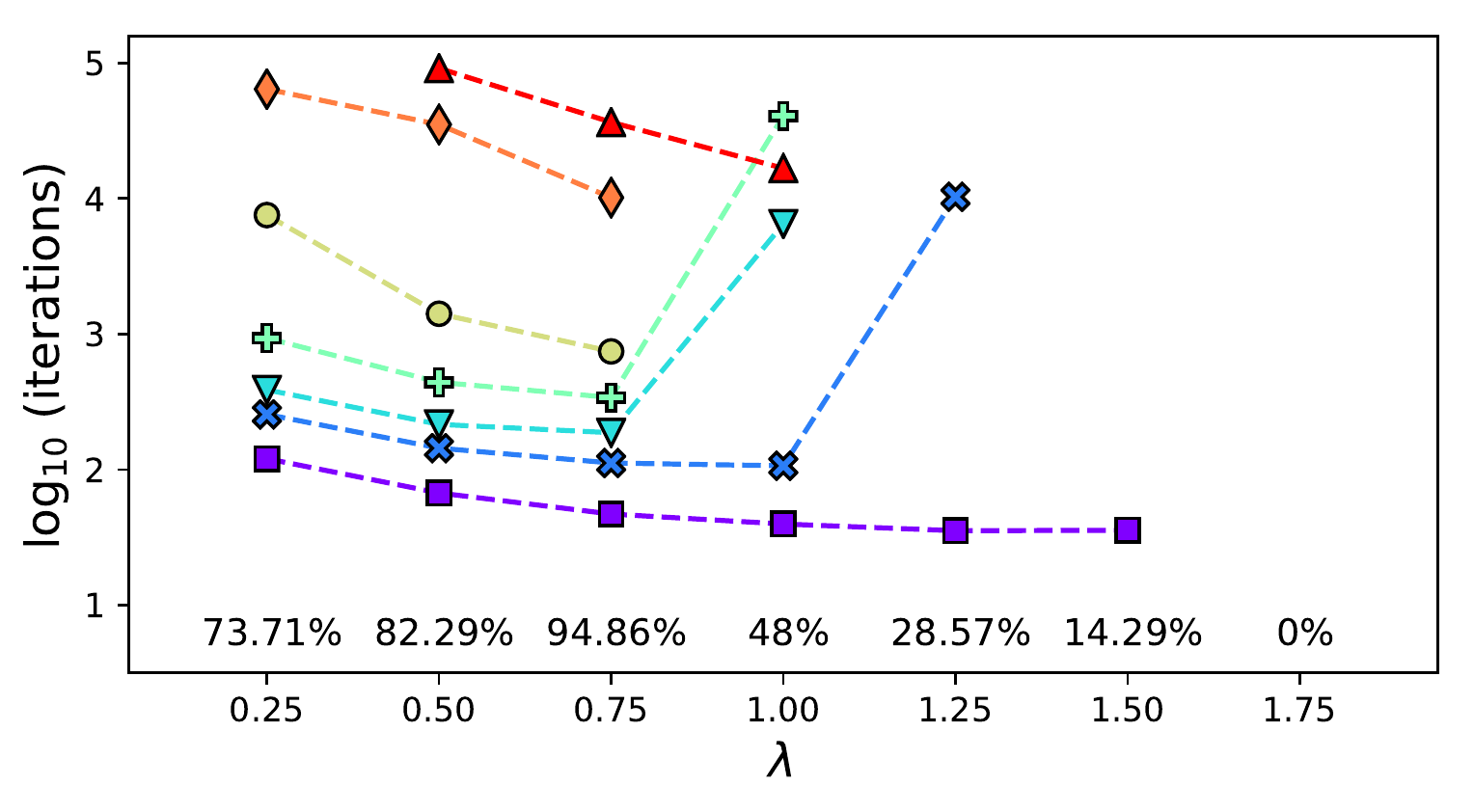}}
	\subfigure[$10$ colors]{\includegraphics[width=0.67\linewidth]{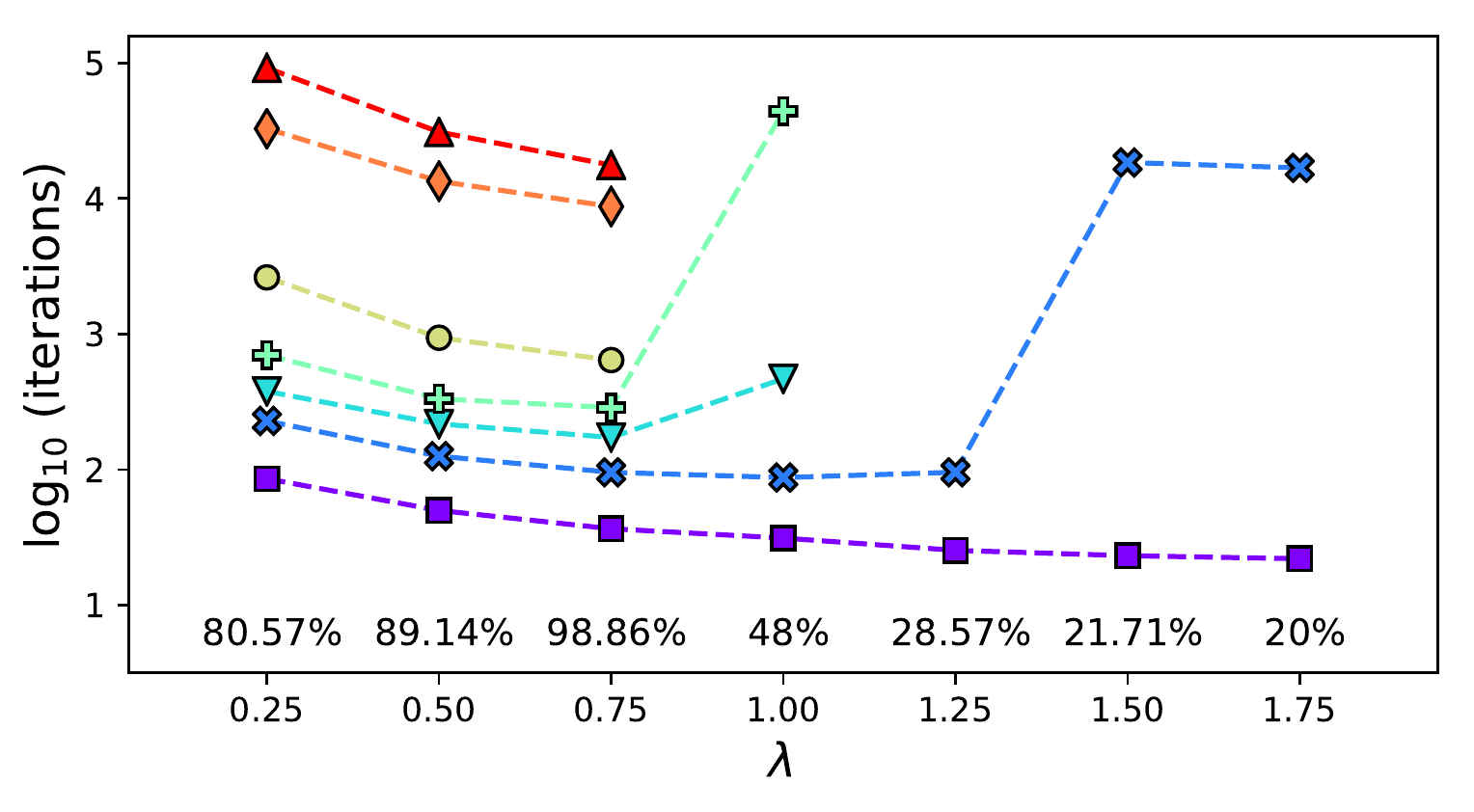}}
	\caption{Results of the experiment on $m$-colorable random graphs for $m=8,9,10$, for the implementation $T_{C_1,C_2,\lambda}$ of DR.  Each marker corresponds to the median of the solved instances among 10 random starting points. At the bottom of each graph we show the percentage of solved instances for each value of $\lambda$. Instances were considered
as unsolved after 100,000 iterations} \label{fig:rand}
\end{figure}

\begin{figure}[ht!]

	\centering
\includegraphics[width=0.67\linewidth]{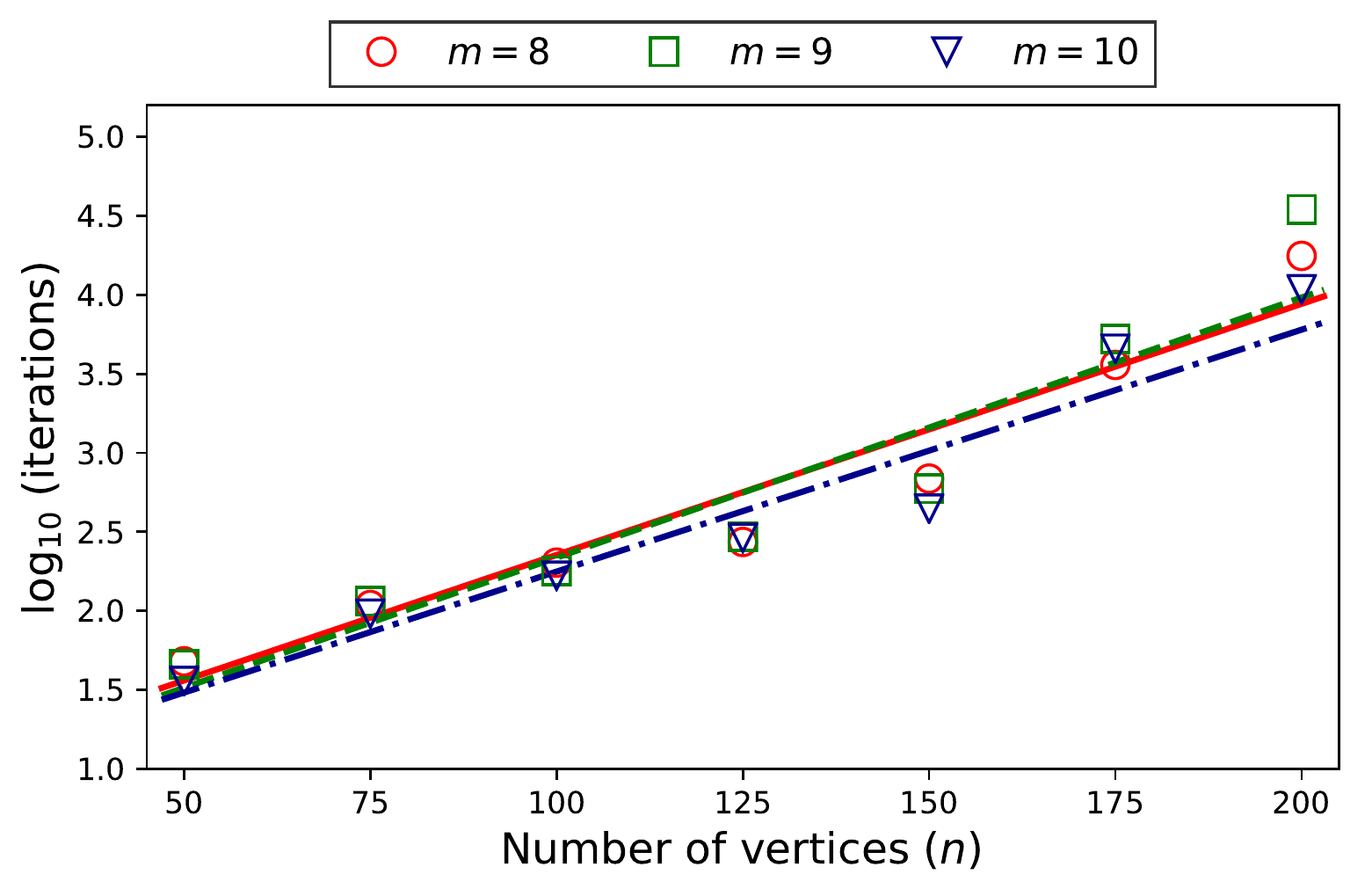}
	\caption{Results of the experiment on $m$-colorable random graphs for $m=8,9,10$, for DR implemented with $T_{C_1,C_2,0.75}$. Each marker corresponds to the median of the solved instances among 10 random starting points, and the lines were obtained by linear regression among all the solved instances}\label{fig:rand_scale}
\end{figure}

\subsection{Windmill graphs}

We turn our attention to a very simple graph for which the binary formulation has trouble finding solutions: the so-called \emph{windmill graph} $\Wd(a,b)$, which is constructed for $a\geq 2$ and $b\geq 2$ by joining $b$ copies of a complete graph with $a$ vertices at a shared vertex. Its chromatic number is $a$, and the graph can be easily colored (there are $a((a-1)!)^b$ different ways to do it). Despite this abundance of valid colorings, all of them are equivalent under a permutation of the colors. Thus, by~\Cref{rmk:equiv_sol}, there exists a unique solution to the rank feasibility problem. This is not the case, however, for the binary formulation.

In our experiments with the binary formulation, the DR algorithm fails to find colorings of windmill graphs rather often. We tentatively attribute this to the high multiplicity of solutions in the binary formulation. In~\cite{AC18} we addressed this problem by augmenting the model with information about the maximal cliques of the graph. A \emph{clique} is a subset of nodes whose induced subgraph is complete; that is, a subset where all the vertices are connected to each other. Further, a clique is \emph{maximal} if it is not contained in a larger clique. The set of maximal cliques would normally be unknown (and difficult to find) for general graphs, so a formulation that does not require this information would generally be preferred.

In our next experiment, whose results are shown in~\Cref{tbl:wd}, we compare the binary formulation with and without maximal clique information, and the rank formulation for coloring sixteen windmill graphs of different parameters.

\newlength{\ancho}
\setlength{\ancho}{.92cm}
\begin{table}[ht!]
\centering
\begin{tabular}{|c|c|| r S[table-format=2.2,table-column-width=\ancho] S[table-format=4,table-column-width=\ancho, group-digits=true] || r S[table-format=2.2,table-column-width=\ancho] S[table-format=4,table-column-width=\ancho] || r S[table-format=2.2,table-column-width=\ancho] S[table-format=4,table-column-width=\ancho] |}
\hhline{--||---||---||---}
\multicolumn{2}{|c||}{\multirow{2}{*}{$\Wd(a,b)$}} & \multicolumn{3}{c||}{\multirow{2}{*}{Binary formulation}} & \multicolumn{3}{c||}{\multirow{1}{*}{Binary formulation}} & \multicolumn{3}{c|}{\multirow{1}{*}{Rank}} \\
\multicolumn{2}{|c||}{} & \multicolumn{3}{c||}{} & \multicolumn{3}{c||}{with clique info.} & \multicolumn{3}{c|}{formulation} \tabularnewline
\hhline{--||---||---||---}
 a & b & Success & {Time} & {Iter.} & {Success} & {Time} & {Iter.} & {Success} & {Time} & {Iter.} \tabularnewline
\hhline{==::===::===::===}
\multirow{4}{*}{5} & 5 & 10/10 & 0.05 & 226 & 10/10 & 0.02 & 63 & 10/10 & 0.01 & 20 \tabularnewline
                   & 10 & 10/10 & 0.13 & 375 & 10/10 & 0.04 & 93 & 10/10 & 0.02 & 33 \tabularnewline
                   & 15 & 9/10 & 0.23 & 503 & 10/10 & 0.06 & 135 & 10/10 & 0.03 & 43 \tabularnewline
                   & 20 & 9/10 & 0.3 & 521 & 10/10 & 0.1 & 170 & 10/10 & 0.05 & 51 \tabularnewline
\hhline{==::===::===::===}
\multirow{4}{*}{10} & 5 & 1/10 & 1.12 & 1886 & 10/10 & 0.12 & 200 & 10/10 & 0.03 & 33 \tabularnewline
                   & 10 & 0/10 & {-} & {-} & 10/10 & 0.27 & 242 & 10/10 & 0.08 & 54 \tabularnewline
                   & 15 & 0/10 & {-} & {-} & 10/10 & 3.47 & 1729 & 10/10 & 0.24 & 86 \tabularnewline
                   & 20 & 0/10 & {-} & {-} & 10/10 & 5.2 & 1531 & 10/10 & 0.45 & 109 \tabularnewline
\hhline{==::===::===::===}
\multirow{4}{*}{15} & 5 & 0/10 & {-} & {-} & 10/10 & 0.54 & 330 & 10/10 & 0.06 & 44 \tabularnewline
                   & 10 & 0/10 & {-} & {-} & 10/10 & 1.69 & 369 & 10/10 & 0.29 & 92 \tabularnewline
                   & 15 & 0/10 & {-} & {-} & 10/10 & 5.21 & 588 & 10/10 & 0.85 & 144 \tabularnewline
                   & 20 & 0/10 & {-} & {-} & 10/10 & 13.35 & 949 & 10/10 & 1.69 & 180 \tabularnewline
\hhline{==::===::===::===}
\multirow{4}{*}{20} & 5 & 0/10 & {-} & {-} & 10/10 & 2.62 & 642 & 10/10 & 0.15 & 68 \tabularnewline
                   & 10 & 0/10 & {-} & {-} & 10/10 & 12.52 & 1059 & 10/10 & 0.63 & 119 \tabularnewline
                   & 15 & 0/10 & {-} & {-} & 10/10 & 16.95 & 729 & 10/10 & 1.83 & 170 \tabularnewline
                   & 20 & 0/10 & {-} & {-} & 8/10 & 31.76 & 828 & 10/10 & 7.3 & 297 \tabularnewline
\hhline{--||---||---||---}
\end{tabular}
\caption{Summary of the results of DR for finding proper colorings of windmill graphs. For each formulation, we show the number of solved instances, the averaged time (in seconds) and the averaged number of iterations. Instances were considered as unsolved after $60$ seconds}
\label{tbl:wd}
\end{table}

We observe that the addition of maximal clique information is crucial for the success of the binary formulation. Without adding it, the DR algorithm was not able to find any solutions for even modestly large values of $a$. On the other hand, the superior performance of the rank formulation for this graph is apparent, both in terms of number of iterations and time. We emphasize again that the rank formulation does not use maximal clique information, and despite this, it achieved a success rate of 100\%.

\subsection{Sudokus}

To test the rank matrix model for precoloring stated in~\Cref{sec:precol_model}, we turn to the Sudoku data set~\texttt{top95}\footnote{\texttt{top95}: \url{http://magictour.free.fr/top95}}, which was the one used in the experiments in~\cite{ABTcomb,AC18} because it contains 95 hard Sudoku instances.

In our first  experiment, we compare the rank formulation (with $T_{C_1,C_2,0.75}$) and the binary formulation for precoloring~\cite[Section~4]{AC18}, with maximal clique information included in the latter. We also compare with the standard divide-and-concur formulation, where solutions for $n\times n$ puzzles are encoded using four copies of $n\times n\times n$ grids of binary indicator variables (see~\cite[Section~6.2]{ABTcomb} for a more detailed explanation), which will be referred to as the \emph{cubic} formulation. For each of these three formulations and each of the $95$ puzzles, the DR algorithm was run from $10$ random starting points. The performance profiles of the results is displayed in~\Cref{fig:sudoku_PP}.

\begin{figure}[ht!]
	\centering
	\includegraphics[width=.62\linewidth]{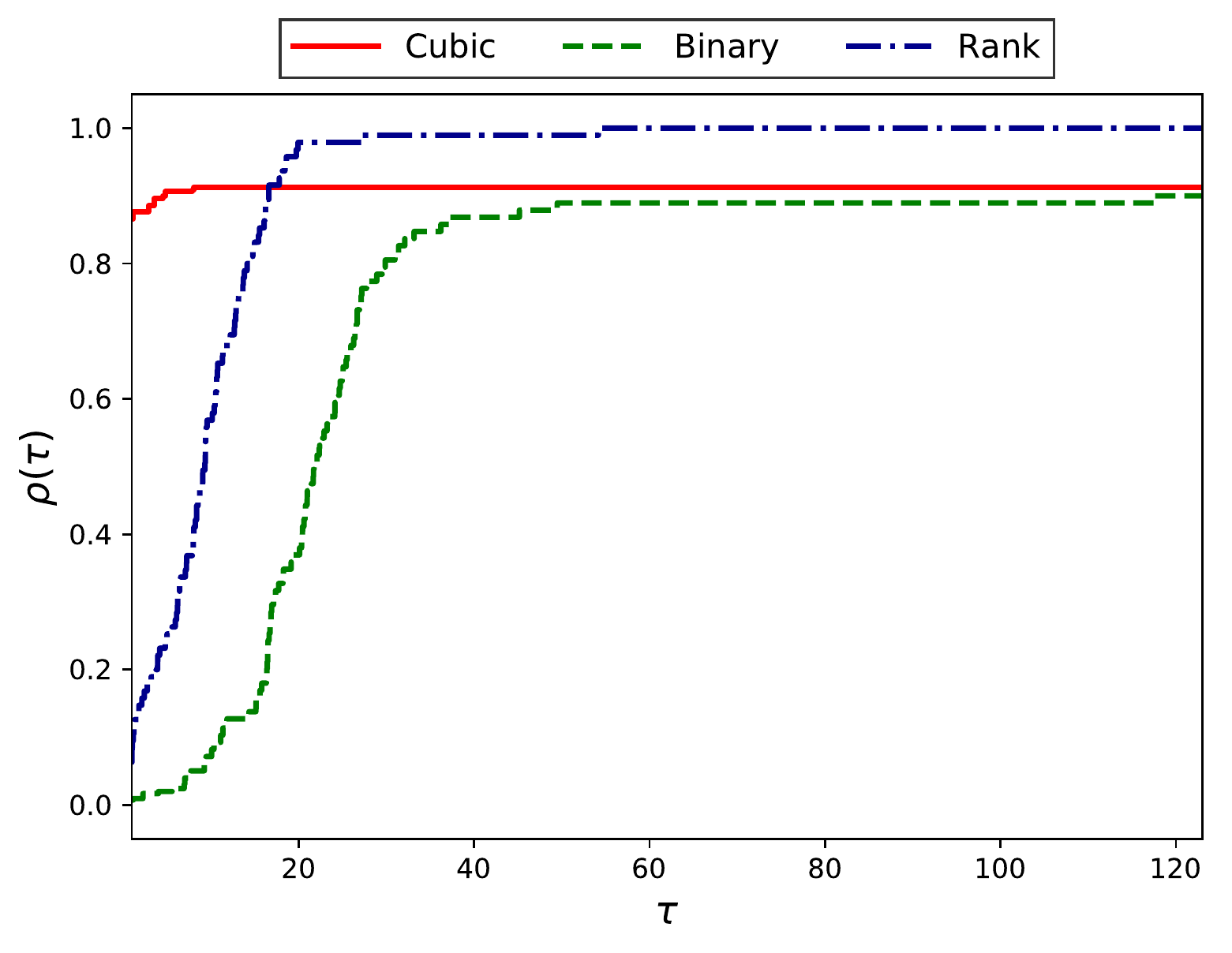}
	\caption{Performance profiles comparing the cubic, binary and rank formulations for solving 95 Sudoku problems. For each problem, 10 starting points were randomly generated. Instances were considered as unsolved after 5 minutes.} \label{fig:sudoku_PP}
\end{figure}

The cubic formulation was the fastest in $86.36$\% of the instances. On average it solved a Sudoku in $4.65$ seconds, while the binary and  rank formulations needed $35.8$ and $13.79$ seconds, respectively. Regarding the success of the algorithm, the cubic and binary formulations solved about $90\%$ of the instances. The rank formulation was the clear winner in terms of success, as it solved every single instance, even for those puzzles in the library on which DR has been observed to be highly unsuccessful (see~\cite[Table~2]{AC18}).

To further challenge the rank formulation, we performed experiments on the so-called `nasty' Sudoku (shown in~\Cref{fig:nasty}). The `nasty' Sudoku has very low success rate in the cubic formulation (see~\cite[Section~6.5]{ABTcomb}), as the algorithm almost always enters a limit cycle (see~\cite[Table~4]{ABTcomb}). This is not the case, however, for the rank formulation. In our next experiment we compare the cubic, binary and rank formulations for solving the `nasty' Sudoku from $100$ random starting points. The results are summarized in~\Cref{fig:nasty}. The rank formulation obtained again a success rate of $100\%$. The second most successful formulation was the binary one, which was only  able to find a solution for $19\%$ of the starting points. So far, we have not been able to find any Sudoku on which the rank formulation failed to find a solution for any starting point.

\setlength{\sudokusize}{3.7cm}\setlength\sudokuthickline{1.2pt}
\renewcommand*\sudokuformat[1]{\small\sffamily #1}
\begin{figure*}
    \begin{center}
\begin{tabular}{cc}
	\begin{minipage}{0.25\textwidth}
		\begin{sudoku}
		 |7| | | | |9| |5| |.
		 | |1| | | | | |3| |.
		 | | |2|3| | |7| | |.
		 | | |4|5| | | |7| |.
		 |8| | | | | |2| | |.
		 | | | | | |6|4| | |.
		 | |9| | |1| | | | |.
		 | |8| | |6| | | | |.
		 | | |5|4| | | | |7|.
		\end{sudoku}
	\end{minipage}
	&
	\begin{minipage}{0.7\textwidth}
	\begin{tabular}{|c||c|c||c|c||c|c|}
		\cline{2-7}
		\multicolumn{1}{c|}{} & \multicolumn{2}{c||}{Cubic} & \multicolumn{2}{c||}{Binary} & \multicolumn{2}{c|}{Rank} \tabularnewline
		\hhline{-||--||--||--}
		Time & Inst. & Cumul. & Inst. & Cumul. & Inst. & Cumul.  \tabularnewline
		\hhline{=::======}
0-24 & 12 & 12\% & 15 & 15\% & 61 & 61\% \tabularnewline
25-49 & 0 & 12\% & 2 & 17\% & 36 & 97\% \tabularnewline
50-99 & 0 & 12\% & 1 & 18\% & 3 & 100\% \tabularnewline
100-299 & 0 & 12\% & 1 & 19\% & 0 & 100\% \tabularnewline
\hhline{-||--||--||--}
Unsolved & 88 & 100\% & 81 & 100\% & 0 & 100\% \tabularnewline
\hhline{-||------}
	\end{tabular}
	\end{minipage}
	\\
	\end{tabular}
	\end{center}
	\caption{Number of solved instances (right), among 100 random starting points, to find the solution of the `nasty' Sudoku (left) by DR with the cubic, binary, and rank formulations. For each interval of time (in seconds), we show the number of solved instances and the cummulative proportion of solved instances for each formulation. The algorithm was stopped after a maximum of 5~minutes, in which case the problem was labeled as ``Unsolved''}\label{fig:nasty}
\end{figure*}

\subsection{DIMACS benchmark instances}

In our final experiment, we test the rank formulation on the widely used graph coloring library from~\texttt{DIMACS benchmark instances}\footnote{\texttt{DIMACS benchmark instances}: \url{http://cse.unl.edu/~tnguyen/npbenchmarks/graphcoloring.html}}. This collection contains various classes of graphs, such as random or quasi-random graphs, problems based on register allocation for variables in real codes, or class scheduling graphs, among others.

The DR algorithm was applied to a wide sample of the aforementioned benchmark instances. Guided by the results in the previous experiments, we used the implementation
$T_{C_1,C_2,0.75}$. For each graph, the algorithm was run from $10$ random starting points and was stopped after a maximum time of one hour. In~\Cref{tbl:dimacs} we present the results of the experiment, as well as the main features of the selected instances. The unsuccessful instances mainly occurred on the very large graphs, on which the algorithm may have succeeded given more time.

\sisetup{group-separator = {,},group-minimum-digits=4,table-number-alignment =right}
\begin{table}[ht!]
\centering\small
\begin{tabular}{l S[table-format=3] S[table-format=4] S[table-format=2] r  S[table-format=4] S[table-format=3.2]}
\hline
\hline
{Instances} & {Vertices} & {Edges} & {Colors} & {Success} & {Iter} & {Time (s)} \tabularnewline
\hline \hline
\texttt{fpsol2.i.1} & 496 & 11654 & 65 & 10/10 & 8984 & 463.94 \tabularnewline
\texttt{fpsol2.i.2} & 451 & 8691 & 30 & 10/10 & 13316 & 495.94 \tabularnewline
\texttt{fpsol2.i.3} & 425 & 8688 & 30 & 10/10 & 14454 & 480.27 \tabularnewline \hline
\texttt{inithx.i.1} & 864 & 18707 & 54 & 10/10 & 16174 & 2443.43 \tabularnewline
\texttt{inithx.i.2} & 645 & 13979 & 31 & 10/10 & 20049 & 1500.45 \tabularnewline
\texttt{inithx.i.3} & 621 & 13969 & 31 & 10/10 & 20604 & 1432.43 \tabularnewline \hline
\texttt{le450\_15a} & 450 & 8168 & 15 & 4/10 & 61365 & 1944.35 \tabularnewline
\texttt{le450\_15b} & 450 & 8169 & 15 & 8/10 & 65537 & 2076.54 \tabularnewline
\texttt{le450\_15c} & 450 & 16680 & 15 & 10/10 & 5464 & 173.1 \tabularnewline
\texttt{le450\_15d} & 450 & 16750 & 15 & 10/10 & 19718 & 619.74 \tabularnewline
\texttt{le450\_25a} & 450 & 8260 & 25 & 10/10 & 1938 & 68.93 \tabularnewline
\texttt{le450\_25b} & 450 & 8263 & 25 & 10/10 & 1849 & 65.82 \tabularnewline
\texttt{le450\_25c} & 450 & 17343 & 25 & 0/10 & {-} & {-} \tabularnewline
\texttt{le450\_25d} & 450 & 17425 & 25 & 0/10 & {-} & {-} \tabularnewline
\texttt{le450\_5a} & 450 & 5714 & 5 & 10/10 & 3071 & 82.47 \tabularnewline
\texttt{le450\_5b} & 450 & 5734 & 5 & 10/10 & 8885 & 238.33 \tabularnewline
\texttt{le450\_5c} & 450 & 9803 & 5 & 10/10 & 3212 & 86.68 \tabularnewline
\texttt{le450\_5d} & 450 & 9757 & 5 & 10/10 & 1644 & 44.49 \tabularnewline \hline
\texttt{mulsol.i.1} & 197 & 3925 & 49 & 10/10 & 2331 & 18.79 \tabularnewline
\texttt{mulsol.i.2} & 188 & 3885 & 31 & 10/10 & 8696 & 63.18 \tabularnewline
\texttt{mulsol.i.3} & 184 & 3916 & 31 & 10/10 & 7814 & 55.88 \tabularnewline
\texttt{mulsol.i.4} & 185 & 3946 & 31 & 10/10 & 8584 & 60.71 \tabularnewline
\texttt{mulsol.i.5} & 186 & 3973 & 31 & 10/10 & 8685 & 62.72 \tabularnewline \hline
\texttt{zeroin.i.1} & 211 & 4100 & 49 & 10/10 & 3014 & 27.1 \tabularnewline
\texttt{zeroin.i.2} & 211 & 3541 & 30 & 10/10 & 4775 & 39.08 \tabularnewline
\texttt{zeroin.i.3} & 206 & 3540 & 30 & 10/10 & 4286 & 34.51 \tabularnewline \hline
\texttt{anna} & 138 & 493 & 11 & 10/10 & 354 & 1.04 \tabularnewline
\texttt{david} & 87 & 406 & 11 & 10/10 & 167 & 0.26 \tabularnewline
\texttt{homer} & 561 & 1628 & 13 & 10/10 & 1222 & 59.01 \tabularnewline
\texttt{huck} & 74 & 301 & 11 & 10/10 & 81 & 0.11 \tabularnewline
\texttt{jean} & 80 & 254 & 10 & 10/10 & 98 & 0.13 \tabularnewline \hline
\texttt{games120} & 120 & 638 & 9 & 10/10 & 109 & 0.24 \tabularnewline \hline
\texttt{miles1000} & 128 & 3216 & 42 & 10/10 & 570 & 2.43 \tabularnewline
\texttt{miles1500} & 128 & 5198 & 73 & 10/10 & 4736 & 24.65 \tabularnewline
\texttt{miles250} & 128 & 387 & 8 & 10/10 & 173 & 0.4 \tabularnewline
\texttt{miles500} & 128 & 1170 & 20 & 10/10 & 307 & 1.07 \tabularnewline
\texttt{miles750} & 128 & 2113 & 31 & 10/10 & 671 & 2.54 \tabularnewline \hline
\texttt{myciel3} & 11 & 20 & 4 & 10/10 & 7 & 0.0 \tabularnewline
\texttt{myciel4} & 23 & 71 & 5 & 10/10 & 15 & 0.0 \tabularnewline
\texttt{myciel5} & 47 & 236 & 6 & 10/10 & 41 & 0.03 \tabularnewline
\texttt{myciel6} & 95 & 755 & 7 & 10/10 & 179 & 0.26 \tabularnewline
\texttt{myciel7} & 191 & 2360 & 8 & 9/10 & 377 & 1.52 \tabularnewline \hline
\texttt{mug88\_1} & 88 & 146 & 4 & 10/10 & 43 & 0.05 \tabularnewline
\texttt{mug88\_25} & 88 & 146 & 4 & 10/10 & 46 & 0.05 \tabularnewline
\texttt{mug100\_1} & 100 & 166 & 4 & 10/10 & 54 & 0.07 \tabularnewline
\texttt{mug100\_25} & 100 & 166 & 4 & 10/10 & 47 & 0.06 \tabularnewline
\hline\hline
\end{tabular}
\caption{Summary of the results of the DR algorithm implemented with $T_{C_1,C_2,0.75}$ for finding proper colorings of a representative sample of DIMACS benchmark instances. For each problem, we show the number of solved runs, the average time (in seconds) and the average number of iterations. We also include the number of nodes and edges, and the chromatic number of each graph. Runs were considered as unsolved after $3600$ seconds}
\label{tbl:dimacs}
\end{table}

\section{Conclusions}

In the emerging field of projection-based heuristic algorithms for solving combinatorially hard problems, the competition is usually framed to be about the choice of operator (DR, ADMM, etc.). In this study, featuring graph vertex coloring with the DR algorithm, we have shown that the choice of constraint formulation has a very significant effect, and in the end may prove to be even more important than the choice of operator. This conclusion comes from numerical experiments demonstrating, over a wide spectrum of instances, the superiority of the rank-constrained matrix formulation~\cite{KMS} over a previously studied formulation based on binary indicator variables.

The failure mechanism of projection-based heuristic algorithms is trapping on limit cycles. Our experiments indicate that the rank-constrained matrix formulation appears to be immune to this problem, achieving 100\% success rates independent of the choice of starting point. Most notable is the success on the so-called `nasty' Sudoku (treated as a graph pre-coloring instance), on which all other known formulations have no better than a $20\%$ success rate.
This approach also does not come at a great cost in implementation, and is able to solve graph coloring instances from the DIMACS benchmark collection with hundreds of vertices and thousands of edges, often in much less than an hour.

We speculate that the good performance of the rank-constrained matrix formulation may be linked to the elimination of the high symmetry-based solution multiplicity of competing formulations. This is strictly an empirical observation, and we have no proposal on how solution multiplicity might be linked to limit cycle behavior. Our results are offered as motivation for pursuing this direction in future research on projection based algorithms.

\paragraph{Acknowledgments}
This work was initiated during the BIRS workshop on Splitting Algorithms, Modern Operator Theory, and Applications, organized by Heinz Bauschke, Regina Burachik and Russell Luke in Oaxaca (Mexico), in 2017. The authors thank the organizers for an excellent meeting and bringing us together.

F.J. Arag\'on and R. Campoy were partially supported by MINECO of Spain  and  ERDF of EU, grant MTM2014-59179-C2-1-P. F.J. Arag\'on was supported by the Ram\'on y Cajal program by MINECO of Spain  and  ERDF of EU (RYC-2013-13327) and R. Campoy was supported by MINECO of Spain and ESF of EU (BES-2015-073360) under the program ``Ayudas para contratos predoctorales para la formaci\'on de doctores 2015''.


\begin{thebibliography}{99}

\bibitem{AF99}
Achlioptas, D., Friedgut, E.:
\newblock A sharp threshold for $k$-colorability.
\newblock Random Struct. Alg. 14, 63--70 (1999)

\bibitem{AM99}
Achlioptas, D., Molloy, M.:
\newblock Almost all graphs with $2.522n$ edges are not $3$-colorable.
\newblock Elec. Jour. Of Comb. 6(1), R29 (1999)

\bibitem{ABMY14}
Arag\'on Artacho, F.J., Borwein, J.M., Mart\'in-M\'arquez, V., Yao, L.:
\newblock Applications of convex analysis within mathematics.
\newblock Math. Program. 148(1--2), Ser. B, 49--88 (2014)

\bibitem{ABTmatrix}
Arag\'on Artacho, F.J., Borwein, J.M., Tam, M.K.:
\newblock Douglas--Rachford feasibility methods for matrix completion problems.
\newblock ANZIAM J. 55(4), 299--326 (2014) 

\bibitem{ABTcomb}
Arag\'on Artacho, F.J., Borwein, J.M., Tam, M.K.:
\newblock Recent results on Douglas--Rachford methods for combinatorial optimization problem.
\newblock J. Optim. Theory. Appl. 163(1), 1--30 (2014) 

\bibitem{ABT16}
Arag\'on Artacho, F.J., Borwein, J.M., Tam, M.K.:
\newblock Global behavior of the Douglas--Rachford method for a nonconvex feasibility problem.
\newblock J. Glob. Optim. 65(2), 309--327 (2016)

\bibitem{AC18}
Arag\'on Artacho, F.J., Campoy, R.:
\newblock Solving graph coloring problems with the Douglas--Rachford algorithm,
\newblock Set-Valued Var. Anal. 26(2), 277--304 (2018)

\bibitem{BBR78}
Baillon, J.B., Bruck, R.E., Reich, S.:
\newblock On the asymptotic behavior of nonexpansive mappings and semigroups in Banach spaces.
\newblock Houston J. Math. 4(1), 1--9 (1978)

\bibitem{BC17}
Bauschke, H.H., Combettes, P.L.:
\newblock Convex analysis and monotone operator theory in Hilbert spaces, 2nd edn.
\newblock Springer, Berlin (2017)


\bibitem{BKroad}
Bauschke, H.H., Koch, V.R.:
\newblock Projection methods: Swiss army knives for solving feasibility and best approximation problems with halfspaces.
\newblock Contemp. Math. 636, 1--40 (2015) 


\bibitem{BNlocal}
Bauschke, H.H., Noll, D.:
\newblock On the local convergence of the Douglas--Rachford algorithm.
\newblock Arch. Math. 102(6), 589--600 (2014) 


\bibitem{benoist}
Benoist, J.:
\newblock The Douglas--Rachford algorithm for the case of the sphere and the line.
\newblock J. Global Optim. 63(2), 363--380 (2015)

\bibitem{C12}
Cegielski, A.:
\newblock Iterative methods for fixed point problems in Hilbert spaces.
\newblock Lecture Notes in Mathematics, 2057. Springer, Heidelberg (2012)

\bibitem{C04}
Chaitin, G.J.:
\newblock Register allocation and spilling via graph coloring.
\newblock SIGPLAN Not., 39(4), 66--74 (2004)

\bibitem{DM02}
Dolan, E.D., Mor\'e, J.J.:
\newblock Benchmarking optimization software with performance profiles.
\newblock Math. Program. 91(2), Ser. A, 201--213 (2002)

\bibitem{Elser}
Elser, V., Rankenburg, I., Thibault, P.:
\newblock Searching with iterated maps.
\newblock Proc. Natl. Acad. Sci. 104(2), 418--423 (2007)

\bibitem{ER59} Erd\"{o}s, P., R\'enyi, A.:
\newblock On random graphs I.
\newblock Publ. Math. Debrecen 6, 290--297 (1959)


\bibitem{FT12}
Formanowicz, P., Tana\'s, K.:
\newblock A survey of graph coloring - its types, methods and applications.
\newblock Foundations of Computing and Decision Sciences, 37(3), 223--238 (2012)

\bibitem{GJS76}
Garey, M.R., Johnson, D.S., So, H.C.:
\newblock An application of graph coloring to printed circuit testing.
\newblock IEEE Transactions on circuits and systems, 23(10), 591--599 (1976)

\bibitem{H80}
Hale, W.K.:
\newblock Frequency assignment: Theory and applications.
\newblock Proceedings of the IEEE, 68(12), 1497--1514 (1980)

\bibitem{HLnonconvex}
Hesse, R., Luke, D.R.:
\newblock Nonconvex notions of regularity and convergence of fundamental algorithms for feasibility problems.
\newblock SIAM J. Optim. 23(4), 2397--2419 (2013)

\bibitem{HJ13}
Horn, R.A., Johnson, C.R.:
\newblock Matrix analysis. Second edition.
\newblock Cambridge University Press, Cambridge (2013)

\bibitem{ISU16}
Izmailov, A.F., Solodov, M.V., Uskov, E.T.:
\newblock Globalizing stabilized sequential quadratic programming method by smooth primal-dual exact penalty function.
\newblock J. Optim. Theor. Appl. 169(1), 1--31 (2016)

\bibitem{JT95}
Jensen, T.R., Toft, B.:
\newblock Graph coloring problems.
\newblock John Wiley \& Sons, New York (1995)


\bibitem{mpmath}
Johansson,  F: \texttt{mpmath}, version 1.0 (2017), \href{http://mpmath.org}{http://mpmath.org}

\bibitem{KMS}
Karger, D., Motwani, R., Sudan, M.:
\newblock Approximate graph coloring by semidefinite programming.
\newblock Journal of the ACM (JACM) 45(2), 246--265 (1998)

\bibitem{K72}
Karp, R.M.:
\newblock Reducibility among combinatorial problems.
\newblock In Complexity of Computer Computations, R. Miller and J. Thatcher, eds., Plenum Press, New York, 85--103 (1972)

\bibitem{L79}
Leighton, F.T.:
\newblock A graph coloring algorithm for large scheduling problems.
\newblock J. Res. Nat. Bur. Standard 84(6), 489--506 (1979)

\bibitem{L16}
Lewis, R.M.R.:
\newblock A Guide to Graph Colouring: Algorithms and Applications.
\newblock Springer International Publishing (2016)

\bibitem{PMX98}
Pardalos, P.M., Mavridou, T., Xue, J.:
\newblock The graph coloring problem: A bibliographic survey.
\newblock In Handbook of combinatorial optimization, Springer US, 1077--1141 (1998)

\bibitem{OEIS} OEIS Foundation Inc.:
\newblock The On-Line Encyclopedia of Integer Sequences (2018), \href{https://oeis.org/A088202}{https://oeis.org/A088202}

\bibitem{PW02} Parks, H.R., Wills, D.C.:
\newblock An Elementary Calculation of the Dihedral Angle of the Regular N-Simplex.
\newblock The American Mathematical Monthly 109(8) (2002), 756--758.

\bibitem{Plinear}
Phan, H.M.:
\newblock Linear convergence of the Douglas--Rachford method for two closed sets.
\newblock Optim. 65(2), 369--385 (2016)

\bibitem{Pierra}
Pierra, G.:
\newblock Decomposition through formalization in a product space.
\newblock Math. Program. 28, 96--115 (1984)


\bibitem{T17} Tam, M.K.:
\newblock Regularity properties of non-negative sparsity sets.
\newblock J. Math. Anal. Appl. 447(2), 758--777 (2017)


\end{thebibliography}
\end{document}